\documentclass[11pt]{amsart}

\usepackage{amsmath,amssymb,epsfig}
\usepackage{graphicx}
\usepackage{fullpage}
\usepackage{amscd}
\usepackage{amsfonts, amssymb} 
\usepackage{stmaryrd}
\usepackage[all]{xy}
\usepackage{upgreek}
\usepackage{latexsym}
\usepackage{cite}
\usepackage[usenames]{color}
\usepackage{xcolor}
\usepackage{datetime2}
\usepackage{url}
\usepackage{boxedminipage}
\usepackage{amsbsy}
\usepackage{framed}
\usepackage[colorlinks=false, linktocpage=true]{hyperref}

\usepackage{enumitem}
\usepackage{bm}
\newcommand{\fg}{\mathfrak g}

\newcommand{\fm}{\mathfrak m}

\newcommand{\cC}{{\mathcal C}}
\newcommand{\cB}{{\mathcal B}}

\newcommand{\cL}{{\mathcal L}}
\newcommand{\cM}{{\mathcal M}}

\newcommand{\cP}{{\mathcal P}}

\newcommand{\cX}{{\mathcal X}}

\newcommand{\N}{\mathbb{N}}
\newcommand{\R}{\mathbb{R}}
\newcommand{\C}{\mathbb{C}}
\newcommand{\Z}{\mathbb{Z}}

\newcommand{\Q}{\mathbb{Q}}

\usepackage{mathrsfs}

\newcommand{\scA}{\mathscr{A}}
\newcommand{\scB}{\mathscr{B}}
\newcommand{\scC}{\mathscr{C}}
\newcommand{\scD}{\mathscr{D}}

\newcommand{\scF}{\mathscr{F}}
\newcommand{\scG}{\mathscr{G}}

\newcommand{\scM}{\mathscr{M}}
\newcommand{\scN}{\mathscr{N}}
\newcommand{\scO}{\mathscr{O}}
\newcommand{\scP}{\mathscr{P}}
\newcommand{\scS}{\mathscr{S}}
\newcommand{\scT}{\mathscr{T}}
\newcommand{\scR}{\mathscr{R}}

\newcommand{\sfC}{\mathsf{C}}

\newcommand{\Euc}{\mathsf{Euc}}
\newcommand{\Set}{\mathsf{Set}}

\newcommand{\Aff}{\mathsf{\cin Aff}}
\newcommand{\cring}{\cin\mathsf{Ring}}

\newcommand{\DiffSp}{\mathsf{DiffSpace}}
\newcommand{\LCRS}{\mathsf{L\cin RS}}
\newcommand{\CGA}{\mathsf{CGA}}
\newcommand{\CDGA}{\mathsf{CDGA}}
\newcommand{\Mod}{\mathsf{Mod}}
\newcommand{\Open}{\mathsf{Open}}
\newcommand{\Pre}{\mathsf{Psh}}
\newcommand{\Psh}{\mathsf{Psh}}
\newcommand{\Sh}{\mathsf{Sh}}
\newcommand{\sh}{\mathsf{sh}}

\DeclareMathAlphabet{\mathpzc}{OT1}{pzc}{m}{it}

\newcommand{\inv}{^{-1}}

\newcommand{\op}[1]{{#1}^{\mbox{\sf{\tiny{op}}}}}
\newcommand{\cin}{C^\infty}
\newcommand{\hd}{\hat{d}}
\newcommand{\bOmega}{\bm{\Upomega}}

\DeclareMathOperator{\supp}{supp}

\DeclareMathOperator{\Der}{Der}
\DeclareMathOperator{\cDer}{C^\infty Der}
\DeclareMathOperator{\Hom}{Hom}
\DeclareMathOperator{\id}{id}

\DeclareMathOperator{\Spec}{Spec}
\DeclareMathOperator{\ev}{\mathsf{ev}}
\DeclareMathOperator{\pr}{\mathsf{pr}}

\numberwithin{equation}{section}

\theoremstyle{definition}
\newtheorem{thm}{Theorem}[section]
\newtheorem{lemma}[thm]{Lemma}
\newtheorem{theorem}[thm]{Theorem}

\newtheorem{proposition}[thm]{Proposition}
\newtheorem{corollary}[thm]{Corollary}
\newtheorem*{corollary*}{Corollary}
\newtheorem*{claim*}{Claim}

\newtheorem{definition}[thm]{Definition}
\newtheorem{remark}[thm]{Remark}
\newtheorem{example}[thm]{Example}
\newtheorem{notation}[thm]{Notation}

\begin{document}
\title{Differential forms on $\cin$-ringed spaces.} 
\author{ Eugene Lerman}

\begin{abstract} We construct a complex of sheaves of differential forms on a
  local $\cin$-ringed space.  The two main classes of spaces we have
  in mind are differential spaces in the sense of Sikorski and
  $\cin$-schemes.  Just as in the case of manifolds the construction
  is functorial. Consequently forms can be integrated over simplices
  and Stokes' theorem holds.

  \end{abstract}
\maketitle
\tableofcontents

\section{Introduction}

This paper is one in a series of papers on differential geometry of
$\cin$-ringed spaces. At the moment two other papers in the series are
\cite{KL} (which constructs flows of derivations on differential
spaces) and \cite{L_cartan} (which, based on the results of this paper,
constructs  Cartan calculus on $\cin$-ringed spaces).

The goal of this paper is to construct a complex of sheaves of
differential forms on local $\cin$-ringed spaces.  The category of
local $\cin$-ringed spaces includes manifolds, manifolds with corners,
differential spaces in the sense of Sikorski \cite{Sn},
$C^\infty$-differentiable space of Navarro Gonz\'{a}lez and Sancho de
Salas \cite{NGSS}, affine $\cin$-schemes of Dubuc \cite{Dubuc} and
Kreck's stratifolds \cite{Kreck}.

The sheaves of commutative differential graded algebras (CDGAs) that
we construct agree with ordinary sheaves of de Rham differential
forms on manifolds with corners.    Our construction of differential forms is
functorial.  Consequently our differential forms  can be integrated over simplexes,
($\cin$-)singular chains and compact manifolds. As a result a version of Stokes's theorem
holds.

On the level of 1-forms our sheaves are the cotangent complexes of
$\cin$-K\"ahler differentials studies by Joyce \cite{Joy}. In the case of differential
spaces our differential forms appear to agree with the differential forms defined
by Mostow \cite{Mostow}.   In the case of stratifolds our differential
forms include Kreck's.

The construction that we carry out in the paper is fairly straightforward (the necessary background is reviewed in the next section).
 Thanks to a theorem of Dubuc and Kock, given 
a $\cin$-ring $\scC$ there exists a $\scC$-module $\Omega^1_\scC$ and a universal derivation 
$d_\scC:\scC \to \Omega^1_\scC$. This is the analogue of K\"ahler differentials in commutative algebra in the setting of $\cin$-rings.   
Similarly to what happens in commutative algebra the differential $d_\scC$ extends to a derivation $d$ of the exterior algebra $\Lambda^\bullet \Omega^1_\scC$, giving us a commutative differential graded algebra  (a CDGA) $(\Lambda^\bullet \Omega^1_\scC, \wedge, d)$.  This algebra is natural in the $\cin$-ring $\scC$.  
Consequently given a local $\cin$-ringed space $(M, \scA)$ we get a presheaf $(\Lambda^\bullet \Omega^1_\scA, \wedge, d)$ of commutative differential graded algebras, which assigns to each open set $U\subseteq M$ the CDGA $(\Lambda^\bullet \Omega^1_{\scA(U)}, \wedge, d)$.  That is,
\[
(\Lambda^\bullet \Omega^1_\scA)(U): = \Lambda^\bullet \Omega^1_{\scA(U)}
\]
for all open subsets $U\subseteq M$.  In general there is no reason for the presheaf $\Lambda^\bullet \Omega^1_\scA$ to be a sheaf.  There is one  notable exception: when $M$ is a manifold with corners and $\scA$ is the sheaf of smooth functions $\cin_M$ on $M$.  In this case $\Lambda^\bullet \Omega^1_{\cin_M}$ agrees with the sheaf $\bOmega^\bullet_{dR, M}$ of ordinary differential forms (see Lemma~\ref{lem:CDGAcorners}).  
More precisely the presheaves  of CDGAs $(\Lambda^\bullet \Omega^1_{\cin_M}, \wedge, d)$ and $(\bOmega^\bullet_{dR, M}, \wedge, d)$ agree.

For a general local $\cin$-ringed space $(M,\scA)$ we sheafify the presheaf  
$\Lambda^\bullet \Omega^1_\scA$ and obtain a sheaf $\bOmega^\bullet
_\scA$ of CDGAs over $(M,\scA)$.\footnote{We slightly  overload the symbol
  $\bOmega$:  $\bOmega^\bullet_{dR, M}$ denotes the sheaf of
  ordinary differential forms on a manifold $M$ while $\bOmega^\bullet _\scA$ denotes
  the sheaf of $\cin$-algebraic differential forms on a local
  $\cin$-ringed space $(M, \scA)$.  When $M$ is a manifold and $\scA =
  \cin_M$ then $\bOmega^\bullet _\scA = \bOmega^\bullet _{dR, M}$,
  which justifies the overload.} Differential forms on manifolds pull back.  So our ``differential forms" should also ``pull back" along the maps of local $\cin$-ringed spaces.   That is, given a map $(f,f_\#): (M, \scA)\to (N, \scB)$  of
 local $\cin$-ringed spaces, we expect to have a map of sheaves
 $\bar{f}: \bOmega^\bullet _\scB \to  f_*(\bOmega^\bullet_\scA)$ of
 CDGAs.  This is indeed the case  --- see Theorem~\ref{thm:8.13}.
 Moreover, if $M$ is a manifold with corners and structure sheaf
 $\scA$ is the sheaf $\cin_M$ of smooth functions on $M$ we get a map
$f^*: \bOmega^\bullet _\scB(N) \to \bOmega^\bullet_{dR}(M) $ of
commutative differential graded algebras.  In other words our
``abstract'' differential forms on local $\cin$-ringed spaces pull
back to familiar differential forms on manifolds with corners.  If
additinally $N$ is a manifold with corners and $\scB = \cin_N$ then
the pullback map $f^*$ is the  pullback map familiar from differential
geometry (Corollary~\ref{cor:5.11} and Remark~\ref{rmrk:5.12}).
 
 The standard geometric $k$-simplex $\Delta^k$ is a manifold with corners.  It makes sense to define a $k$-simplex of a local $\cin$-ringed space $(M, \scA)$ to be a map
 $\sigma: (\Delta^k, \cin_{\Delta^k})\to (M,\scA)$ of 
$\cin$-ringed spaces.   We define the integral $\int_\sigma \omega$ of a global section  $\omega\in \bOmega^k_\scA
(M)$ over a $k$-simplex $\sigma$ to be the integral $\sigma^*
\omega$ over  the standard simplex $\Delta^k$:
\[
  \int_{\sigma} \omega := \int_{\Delta^k} \sigma^*\omega.
\]
This definition easily extends to singular chains and then a version
of Stokes's theorem holds: for a local $\cin$-ringed space $(M,\scA)$,
a global $\cin$-algebraic de Rham form $\gamma\in
\bOmega^{k-1}_\scA(M)$ and a $k$-simplex $\sigma: (\Delta^k,
\cin_{\Delta^k})\to (M,\scA)$

\[
\int_\sigma d\gamma = \int_{\partial \sigma} \gamma.
\]
Here and elsewhere in the paper the boundary $\partial \sigma$ of a
simplex $\sigma$ is the singular chain
$\partial \sigma:= \sum (-1)^i \sigma \circ d_i$ where
$d_i:\Delta^{k-1} \to \Delta ^k$ are the standard (geometric) face
maps.

\mbox{}\\[4pt]
\noindent
{\bf Acknowledgements} I am grateful to Yael Karshon for years
of conversations throughout the pandemic.  I thank Reyer Sjamaar, Rui
Loja Fernandes, Jordan Watts, Dan Berwick-Evans and Casey Blacker for
useful discussions.

\section{Background} \label{sec:bgrd}
In this section we collect all the mathematical background we need for constructing sheaves of differential forms on $\cin$-ringed spaces.
Most, if not all, of the material is probably known to various experts.  However, I don't know of any source that has all the material in one 
place.

\begin{remark}
In the paper a manifold is always $\cin$, Hausdorff and second
countable.
\end{remark}
\subsection{$\cin$-rings} \quad
Roughly speaking, a {\sf  $\cin$-ring } is a (nonempty)
set $\scC$ together with
operations (i.e., functions) \label{page:c-ring}
\[
g_\scC:\scC^m\to \scC
\]
for all $m$ and all $g\in C^\infty (\R^m)$
such that for all
$n,m\geq 0$, %
 all $g\in C^\infty(\R^m)$ and all
$f_1, \ldots, f_m\in C^\infty (\R^n)$
\begin{equation} \label{eq:2.assoc}
({g\circ(f_1,\ldots, f_m)})_\scC (c_1,\ldots, c_n) =
g_\scC({(f_1)}_\scC(c_1,\ldots, c_n), \ldots, {(f_m)}_\scC(c_1,\ldots, c_n))
\end{equation}
for all $(c_1, \ldots, c_n) \in \scC^n$.

A formal definition of a $\cin$-ring is that of set-valued
product-preserving functor from the category $\Euc$ of Euclidean (also
known as Cartesian) spaces.  This definition requires a certain amount
of unpacking. We start by defining the category $\Euc$.

\begin{definition}[The category $\Euc$ of Euclidean spaces]
  The objects of the category $\Euc$ are the coordinate vector spaces
  $\R^n$, $n\geq 0$,  thought of as $\cin$ manifolds.  The morphisms are
  $C^\infty$ maps.
\end{definition}
We note that all objects of $\Euc$ are finite powers of one object:
$ \R^n = (\R^1)^n$ for all $n\geq 0$.

\begin{definition}\label{def:C2}
A {\sf $C^\infty$-ring} $\cC$ is a functor
\[
\cC:\Euc \to \Set
\]
from the category $\Euc$ of Euclidean spaces to the category $\Set$ of
sets that {\em preserves finite products}.

A {\sf morphism} of $\cin$-rings from a $\cin$-ring $\cC$ to a
$\cin$-ring $\cB$ is a natural transformation $\cC\Rightarrow \cB$.
\end{definition}

\begin{remark} We unpack 
Definition~\ref{def:C2}.
\begin{enumerate}
\item For any $n\geq 0$, the set $\cC(\R^n)$ is an $n$-fold product of
  the set $\scC: = \cC(\R^1)$, i.e., $\cC(\R^n) =\scC^n$.  Moreover
  the canonical projection $\pr_i: \scC\, ^n \to \scC$
  on the $i$th factor is $\cC(\R^n\xrightarrow{x_i} \R) $, where $x_i$
  is the $i$th coordinate function.
\item For any smooth function $g= (g_1,\ldots, g_m):\R^n \to \R^m$ we
  have a map of sets
\[
\cC(g): \cC(\R^n) = \scC^n \to \scC^m = \cC(\R^m)
\]
with
\[
\pr_j \circ \,\cC (g) = \cC(g_j).
\]
\item For any triple of smooth functions $\R^n\xrightarrow{g} \R^m
  \xrightarrow{f} \R$ we have three maps of sets $\cC(g)$, $\cC(f)$,
  $\cC(f\circ g) $ with
\[
\cC(f) \circ \cC(g) = \cC(f\circ g).
\]
\item $\cC(\R^0)$ is a single point set $*$.  This is because by definition
  product-preserving functors take terminal objects to terminal objects.
\end{enumerate}
It follows that $\scC = \cC(\R^1)$ is a set together with, for every
$n\geq 0$, a collection of $n$-ary operations indexed by the set of smooth
functions $\cin(\R^n)$. 
Conversely given a set $\scC$ with a family of operations indexed by 
smooth functions there is a unique product-preserving functor $\cC:\Euc\to
\Set$ with $\cC(\R^1) = \scC$.
\end{remark}

\begin{remark}

Throughout the paper we will treat $\cin$-rings as sets with
operations and morphisms of $\cin$-rings as operation-preserving maps
(rather than as functors and natural transformations).  This treatment
is standard, cf.\ \cite{MR} or \cite{Joy}.

From this point of view ($\cin$-rings are sets with operations) a {\sf
morphism} of $\cin$-rings $\varphi:\scC \to \scD$ is a map of sets
$\varphi$ which preserves all the operations: for any $n$, any $h\in
\cin(\R^n)$ and any $c_1, \ldots, c_n\in \scC$
\[
 \varphi (h_\scC(c_1,\ldots, c_n)  = h_\scD (\varphi(c_1), \ldots,
 \varphi(c_n)). 
\]
    
\end{remark}

\begin{remark} \label{rmrk:cin-R-alg}
Any $\cin$-ring has an underlying $\R$-algebra.  This can be
seen as follows.  The category $\Euc$ contains the subcategory $\Euc_{poly}$
with the same objects whose morphism are polynomial maps.  A
product-preserving functor from $\Euc_{poly}$ to $\Set$ is an
$\R$-algebra.  Clearly any product preserving functor $\cC:\Euc \to
\Set$ restricts to a product preserving functor from $\Euc_{poly}$ to $\Set$.

The structure of an $\R$-algebra on a $\cin$-ring $\scC$ can also be
seen on the level of sets and operations. 
Note that $\R^0 =\{0\}$, $\cin(\R^0) = \R$ and that for any set $\scC$
the zeroth power $\scC^0$ is a 1-point set $*$.  Consequently for
a  $\cin$-ring $\scC$, $n=0$, and $\lambda \in C^\infty (\R^0) = \R$
the corresponding 0-ary operation is a function $\lambda_\scC: *\to
\scC$, which we identify with an element of $\scC$.  By abuse of
notation we denote this element by $\lambda_\scC$.  This gives us a
map $\R\to \scC$, $\lambda \mapsto \lambda_\scC$.  One can show that
unless $\scC$ consists of one element (i.e., $\scC$ is the zero $\cin$-ring)
the map is injective.    Together with the operations $+:\scC^2\to
\scC$ and $\cdot: \scC^2 \to \scC$, which correspond to the functions
$h(x,y) = x+y$ and $g(x,y) = xy$, respectively, the 0-ary operations make any
$\cin$-ring into a unital $\R$-algebra.  We will not notationally
distinguished between a $\cin$-ring and the corresponding (underlying) $\R$-algebra.
\end{remark}

\begin{notation}
$\cin$-rings and their morphisms form a category that we denote by
$\cring$.
\end{notation}

\begin{example} \label{ex:2.1}
Let $M$ be a $\cin$-manifold and $\cin(M)$ the set of
smooth (real-valued) functions. Then $\cin(M)$, equipped with the 
operations
\[
g_{\cin(M)} (a_1,\ldots, a_m) := g\circ (a_1,\ldots, a_m),
\]
(for all $m$, all $g\in \cin(\R^m)$ and all $a_1,\ldots, a_m \in \cin(M)$) is a $\cin$ ring.
\end{example}

\begin{example}
  The real line $\R$ is a $\cin$-ring:   the $m$-ary operations are given by
\[  
  g_\R(a_1,\ldots, a_m):= g(a_1,\ldots, a_m)
\]
for all $g\in \cin(\R^m)$ and all $a_1,\ldots, a_m\in \R$.   Note that
this example  is consistent with Example~\ref{ex:2.1}: if $M$ is a
point $*$, then $\cin(M) = \R$ (since in this case the elements of $\cin(M)$ are functions $*\to
\R$) and composing a function $g\in \cin(\R^m)$ with an $m$-tuple of
constant functions is the same as evaluating $g$ on the constants.
\end{example}  

\begin{example} \label{ex:2.1+}
Let $M$ be a topological space  and $C^0(M)$ the set of
continuous real-valued  functions. Then $C^0(M)$, equipped with the 
operations
\[
g_{\cin(M)} (a_1,\ldots, a_m) := g\circ (a_1,\ldots, a_m),
\]
is also a  $\cin$ ring.
\end{example}

\begin{definition} \label{def:free}
A $\cin$-ring  $\scC$ is {\sf free} on a subset $X \subset \scC$ or is
{\sf freely generated} by a subset $X$ if for
any $\cin$-ring $\scB$ and any map of sets $f:X\to \scB$ there is a
unique map of $\cin$-rings $\tilde{f}:\scC \to \scB$ so that
$\tilde{f} |_X = f$.
\end{definition}

The following fact is well-known.  See, for example, \cite{MR}.
\begin{lemma}\label{lem:rn-is-free}
  The $\cin$-ring $\cin(\R^n)$ of smooth functions on $\R^n$ is a free
  ring on $n$ generators. The generators are the coordinate functions
  $x_1,\ldots, x_n:\R^n\to \R$.
\end{lemma}
\begin{proof}
 For any $\cin$-ring $\scC$, any map
$\varphi:\cin(\R^n)\to \scC$ of $\cin$-rings and any $f\in \cin(\R^n)$
\[
\varphi(f) = \varphi(f_{\cin(\R^n)} (x_1,\ldots, x_n) ) =
f_\scC(\varphi(x_1), \ldots, \varphi(x_n)).
  \]
\end{proof}

\begin{definition}
An {\sf ideal} in a $\cin$-ring $\scC$ is an ideal $I$ in the
$\R$-algebra underlying $\scC$ (cf.\ Remark~\ref{rmrk:cin-R-alg}).
\end{definition}  

\begin{lemma}\label{lem.1}
For any ideal $I$ in a $C^\infty$-ring $\scC$ the quotient $\R$-algebra $\scC/I$ is naturally a $C^\infty $-ring.
\end{lemma}
\begin{proof}
See \cite[Proposition~1.2, p.~18]{MR}.
\end{proof}

\begin{definition}\label{def:fg}
A $\cin$-ring $\scC$ is {\sf finitely generated} if there is a
surjective $\cin$-ring map $\varphi:\cin(\R^n) \to \scC$ for some
$n>0$.  Equivalently $\scC$ is finitely generated if $\scC$ is
isomorphic to $\cin(\R^n)/I$ for some ideal $I\subset \cin(\R^n)$.
\end{definition}
\begin{remark}
  There are  $\cin$-rings that are not finitely generated and they are
  not exotic.   For example we may take the ring $\scC$ to be the differential structure on the set 
  $\Q$ of rational induced by the inclusion $\Q\hookrightarrow \R$.
  See Definition~\ref{def:induced_diff_str} and subsequent
  discussion.  Showing that the induced differential structure on $\Q$
  is not finitely generated is not hard.  We omit the proof.
\end{remark}

We next define modules over $\cin$-rings.  Again, the definition is standard.

\begin{definition} \label{def:2.6}
A {\sf module} over a $\cin$-ring $\scC$ is a  module over
the $\R$-algebra underlying $\scC$ (cf.\ Remark~\ref{rmrk:cin-R-alg}).
\end{definition}  
Before trying to justify this definition,  we give three examples of modules.

 \begin{example}
Let $E\to M$ be a (smooth) vector bundle over a manifold $M$.  Then
the set of sections $\Gamma(E)$ of $E\to M$ is a module over the $\cin$-ring $\cin(M)$.
 \end{example}

 \begin{example} \label{ex:points-module}
Let $(M,\cin(M))$ be a manifold and $p\in M$ a point.  The
evaluation at $p$
\[
  ev_p: \cin(M) \to \R,\qquad ev_p(f) = f(p)
\]
  makes
$\R$ into a $\cin(M)$-module.  Note that the $\cin(M)$-module
structure on $\R$ depends on the choice of a point $p$.
\end{example}

\begin{example}
Any $\cin$-ring $\scC$ is a module over itself.
\end{example}

\noindent
There are several reasons why Definition~\ref{def:mod}  is 
reasonable.  We list three:\\[10pt]
\noindent {\bf (i)}\quad As we have seen above,  if $M$ is a manifold
and $E\to M$ a vector bundle over $M$ then
the vector space $\Gamma(E)$ of sections of the bundle is a module
over the $\R$-algebra $\cin(M)$.  It is reasonable to consider
it as a module over the $\cin$-ring $\cin(M)$.  In particular if $E$
is the tangent bundle $TM$ then the space of sections $\Gamma (TM)$,
which  is the space of vector fields on $M$, is a
module over $\cin(M)$.   And if $E$ is the cotangent bundle $T^*M$ then the
space of sections $\Gamma(E)$ is the space of (ordinary) 1-forms $\bOmega^1_{dR}(M)$ on
$M$, which is also  a module over $\cin(M)$.\\

\noindent {\bf (ii)}\quad \label{ii} Given a $\cin$-ring $\scC$ and an $\scC$-module $\scM$,  the product $\scC\times \scM$ has the structure of a
$\cin$-ring.  Namely, for any $n>0$, any $f\in \cin(\R^n)$ and any
$(a_1,m_1),\ldots (a_n,m_n)\in \scC\times \scM$
\[
f_{\scC\times \scM} ((a_1,m_1),\ldots, (a_n,m_n)) := (f_\scC(a_1,
\ldots, a_n), \sum_{i=1}^n (\partial_i f)_\scC(a_1,\ldots, a_n) \cdot m_i).
  \]
Recall that the multiplication in the $\R$-algebra underlying a $\cin$-ring
comes from the function $f(x,y) = xy\in
\cin(\R^2)$.  Consequently for any $(a_1, m_1), (a_2,m_2)\in \scC\times \scM$
\[
  \begin{split} 
(a_1, m_1)\cdot (a_2, m_2) &= f_{\scC\times \scM} ((a_1, m_1),
(a_2,m_2)) \\&= ( f_{\scC}(a_1,a_2), \sum_{i=1}^2 (\partial_i f)_\scC(a_1,
a_2) \cdot m_i) =
(a_1a_2, a_2m_1 + a_1m_2).
\end{split}
 \] 
 In particular
 \[
   (0, m_1) \cdot (0, m_2) = (0\cdot 0, 0m_1 +0m_2) = (0,0).
\]
Therefore, just as in the case of commutative rings and commutative
$\R$-algebras, modules over $\cin$-rings correspond to square zero
extensions.\\

\noindent {\bf (iii)}\quad \label{iii}  A module $\scM$ over a
$\cin$-ring $\scC$ is a Beck module \cite{Beck, Barr}:  just as in the
case of commutative rings given a $\cin$-ring $\scC$ and an
$\scC$-module $\scM$ the projection $p:\scC\times \scM\to \scC$ on the
first factor is an abelian group object in the slice category
$\cring/\scC$ of $\cin$-rings over $\scC$.  We omit the proof of this claim.\\

\begin{definition}[The category $\Mod$ of modules over $\cin$-rings] \label{rmrk:mod} \label{def:mod}
Modules over $\cin$-rings form a category; we denote it by $\Mod$.   The objects of
$\Mod$ are pairs $(R, M)$ where $R$ is a $\cin$-ring and $M$ is a
module over $R$.   A {\sf morphism} in $\Mod$ from $(R_1, M_1)$ to
$(R_2, M_2)$ is a pair $(f,g)$ where $f:R_1\to R_2$ is a map of
$\cin$-rings and $g:M_1\to M_2$ is a map of abelian groups so that
\[
g(r\cdot m) = f(r)\cdot g(m)
\]
for all $r\in R_1$ and $m\in M_1$.
\end{definition}

\begin{definition}  \label{def:der2} %
  A {\sf $\cin$-derivation  of a $\cin$-ring $\scC$
    with values in a $\scC$-module $\scM$ } is a map 
$X:\scC\to \scM$ so
that for any $n>0$, any $f\in \cin(\R^n)$ and any $a_1,\ldots, a_n\in
\scC$
\begin{equation}\label{eq:der}
X(f_\scC(a_1,\ldots,a_n)) = \sum_{i=1}^n (\partial_i f)_\scC(a_1,\ldots, a_n)\cdot X(a_i).
 \end{equation} 
\end{definition}

\begin{example}
Let $M$ be a smooth manifold. A $\cin$-derivation $X:\cin(M) \to
\cin(M)$ of the $\cin$-ring of smooth 
functions $\cin(M)$ with values in $\cin(M)$ is an ordinary vector
field.

The exterior derivative $d: \cin(M) \to \bOmega_{dR}^1(M)$ is a 
$\cin$-derivation of $\cin(M)$ with values in the module $\bOmega_{dR}^1(M)$ of
ordinary 1-forms.

Let $x\in M$ be a point and $\R$ a $\cin(M)$-module with the module
structure coming from the evaluation map $ev_x:\cin(M)\to \R$, cf.\
Example~\ref{ex:points-module}.  Then a derivation of $\cin(M)$ with
values in the module $\R$ is a tangent vector at the point $x\in M$.
\end{example}

\begin{notation}
The collection of all $\cin$-derivations of a $\cin$-ring $\scC$ with values
in an $\scC$-module $\scM$ is a $\scC$-module.  We denote it by
$\cDer(\scC, \scM)$.  In the case where $\scM = \scC$ we write
$\cDer(\scC)$ for $\cDer(\scC, \scC)$.
\end{notation}

\begin{remark} Recall that a $\cin$-ring $\scC$ has an underlying
  $\R$-algebra structure (Remark~\ref{rmrk:cin-R-alg}).
Therefore given a $\scC$-module $\scM$ we may also consider the space
$\Der(\scC, \scM)$ of $\R$-algebra derivations.  It is easy to see
that any $\cin$-ring derivation of $\scC$ with values in $\scM$ is an
$\R$-algebra derivation (apply a $\cin$-ring derivation to the
operation defined by $f(x,y) = xy\in \cin(\R^2)$ to get the
product rule, for example).   In general $\Der(\scC, \scM)$ is
strictly bigger than $\cin\Der(\scC, \scM)$.  For example,  Osborn shows in \cite[Corollary to Proposition 1]{Osborn} that if $\scC = \cin(\R^n)$ and 
$\scM =\Omega^1_{\cin(\R^n), alg}$ is the module of ordinary  K\"ahler differentials (in the sense of commutative algebra), then 
the universal derivation $d: \cin(\R^n) \to \Omega^1_{\cin(\R^n), alg}$ is {\em not} a $\cin$-derivation
(it is only a $\cin$-derivation on algebraic functions).   

However when the $\cin$-ring in question is a differential structure
$\scF$ on some differential space (see Definition~\ref{def:sikorski}
below) and $\scM = \scF$, then
$\cin\Der(\scF) = \Der( \scF)$.  This follows from a
theorem of Yamashita \cite{Yamashita}  because the $\cin$-ring $\scF$ is
point-determined (see Definition~\ref{def:pd} and Example~\ref{ex:ds_pd}).
\end{remark}

\begin{notation} Given two modules $\scM, \scN$ over the same $\cin$-ring
  $\scC$ we denote the collection of maps of modules from $\scM$ to
  $\scN$ by $\Hom(\scM, \scN)$.  Note that $\Hom(\scM, \scN)$ is,
  again, a module over $\scC$.
\end{notation}

\subsection{Local $\cin$-ringed spaces} \quad
Our next step is to define local $\cin$-ringed spaces and (graded) modules over
local $\cin$-ringed spaces.  We then discuss in more detail two
common examples of local $\cin$-ringed spaces: differential spaces in
the sense of Sikorski and affine $\cin$-schemes of Dubuc.

\begin{remark} \label{rmrk:values}
There are two ways to view
sheaves on a space $M$ with values in an algebraic category $\sfC$.
We can view them as functors $\op{\Open(M)}\to \sfC$, where $\Open(M)$
is the poset of open subsets of $M$ ordered by inclusion.   Or we can
view them as ``$\sfC$-objects'' internal to the category of set-valued
sheaves on $M$.  The latter usually amounts to having a collection of
maps of set-valued sheaves subject a number of commutative diagrams.
We will go back and forth between the two points of view.
\end{remark}

\begin{definition}
An  {\sf $\R$-point} of a $\cin$-ring $\scC$ is a nonzero homomorphism
$p:\scC\to \R$ of $\cin$-rings.
\end{definition}  

\begin{example}
Let $M$ be a manifold  and $\scC = \cin(M)$ the $\cin$-ring of smooth
functions on $M$.  Then for any point $x\in M$ the evaluation map at
$x$
\[
ev_x:\cin(M)\to \R,\qquad f\mapsto f(x)
\]
is an $\R$-point of the $\cin$-ring $\cin(M)$.

Conversely, given an $\R$-point $p:\cin(M) \to \R$ there is a point $x\in M$
so that $ev_x = p$.  This fact  is known as ``Milnor's exercise.''  It is a
theorem of Pursell \cite{Pursell}.

Thus for a manifold $M$ the set of
$\R$-points of $\cin(M)$  is (in bijection with) the set of ordinary
points of $M$.  
\end{example}

\begin{remark}
  The zero $\cin$-ring (which corresponds to the functor
  $\Euc\to \Set$ that sends every object to a 1-point set and whose
  underlying $\R$-algebra is 0) has no $\R$-points.
\end{remark}
\begin{remark}
  The kernel of an $\R$-point $p: \scC \to \R$ of a $\cin$-ring $\scC$
  is a maximal ideal in $\scC$.
\end{remark}
  
\begin{remark}
  There are $\cin$-rings $\scC$  containing maximal  ideals
  $\mathfrak{M}$ with the property that  the field $\scC/\mathfrak{M}$ is not
  isomorphic to $\R$.  Here is an example.  Fix $n>0$. The set of
  compactly supported functions $\cin_c (\R^n)$ is an ideal in the
  $\cin$-ring $\cin(\R^n)$.  Let $\mathfrak{M}$ be a maximal ideal
  containing $\cin_c (\R^n)$ (which exists by Zorn's lemma).  We will
  argue that the quotient $\cin(\R^n)/\mathfrak{M}$ has no $\R$-points
  hence, in particular, cannot be isomorphic to $\R$.

  Consider the sequence of projections $\cin(\R^n) \xrightarrow{\pi}
  \cin(\R^n)/\cin_c (\R^n) \xrightarrow{\pr} \cin(\R^n)/\mathfrak{M}$.
  If $\varphi: \cin(\R^n)/\mathfrak{M} \to \R$ is an $\R$ point, then
  so is the composite $\varphi \circ \pr \circ \pi: \cin(\R^n)\to
  \R$.  By Purcell's theorem mentioned above, there is a point $x\in
 \R^n$ so that $\varphi \circ \pr \circ \pi $ is the evaluation at $x$:
 \[
ev_x = \varphi \circ \pr \circ \pi .
\]    
But then every compactly supported function on $\R^n$ vanishes at $x$,
which is impossible.
\end{remark}

\begin{definition} \label{def:local_ring}
A $\cin$-ring is {\sf local} if it has exactly one $\R$-point.
\end{definition}  

\begin{remark}
Our terminology in Definition~\ref{def:local_ring} follows Dubuc
\cite{Dubuc} and Joyce \cite{Joy}.  Moerdijk and Reyes call our local
$\cin$-rings {\em pointed} local $\cin$-rings.
\end{remark}  

\begin{example} \label{ex:2.36}
Let $M$ be a manifold.  The stalk  $\cin_{M,x}$of the sheaf $\cin_M$ of smooth
functions on $M$ at a point $x$ consists of germs of functions at
$x$.   It is a $\cin$-ring (cf.\ \cite[p.\ 17]{MR}) and it is a local $\cin$-ring:
the unique $\R$-point  $p:\cin_{M,x}\to \R$ is the evaluation at the point $x$.  
See also Lemma~\ref{lem:2.58} below, which proves a more general result.
\end{example}  

%

\begin{definition} \label{def:loc-ringed_space}
A {\sf local $\cin$-ringed space}
is a pair $(M,\scA)$ where $M$ is a topological space and
$\scA$ is a sheaf of $\cin$-rings with the additional property that
the stalks of the sheaf are local $\cin$-rings.   We will refer to the
sheaf $\scA$ as the {\sf structure sheaf} of $(M,\scA)$.
\end{definition}

\begin{example}
It follows from Example~\ref{ex:2.36} that a manifold $M$ together
with the sheaf $\cin_M$ of smooth functions is a local $\cin$-ringed
space.  
\end{example}

\begin{remark} A reader unfamiliar with ringed spaces of any kind may wonder why
 anyone would single out local ringed spaces.   There is a variety of
 reasons.  One of them is that if $(M, \scA)$ is a local
 $\cin$-ringed space,  then for any open set $U\subset M$ a section
 $s\in \scA(U)$ gives rise to a real-valued function $f_s$ as follows:
 given $x\in U$, there is  a unique $\R$-point $\hat{x}: \scA_{x}\to
 \R$.  We then define
 \[
f_s(x) := \hat{x} (s_x), 
 \]  
where, as before, $s_x$ is the germ of $s$ at $x$.   For differential
spaces (which are defined below in Definition~\ref{def:sikorski}), the function $f_s$ is just $s$.   But in general the sheaf $\scA$ need
not be a sheaf of $\cin$-rings of {\em functions},  and locality gives us a
map from local sections to (local) functions.   
See the proof of Theorem~\ref{thm:8.2}  for an example of the use
of this  fact.
\end{remark}  

If manifolds were the only examples of local $\cin$-ringed spaces,
this paper would not have much of a point.   Fortunately there are two
broad classes of examples: affine $\cin$-schemes of Dubuc and
differential spaces in the sense of Sikorski.   They are discussed
below, but not immediately.   %

When discussing derivations of $\cin$-rings  we'll need the
definition of a point-determined $\cin$-ring, which we presently
recall.

\begin{definition} \label{def:pd}
A $\cin$-ring $\scC$ is {\sf point-determined} if $\R$-points separate
elements of $\scC$.  That is, if $a\in\scC$ and $a\not = 0$ then
there is an $\R$-point $p:\scC \to \R$ so that $p(a)\not = 0$.
\end{definition}

\begin{example}
Let $M$ be a topological space.  Then the $\cin$-ring $C^0(M)$ of continuous functions on $M$ is point determined: a function $f:M\to \R$ 
is nonzero if and only if $\ev_x(f) = f(x) \not = 0$ for some point $x\in M$.
\end{example}

\begin{remark} \label{rmrk:2.41-7}
In order to define maps (or, if you prefer, morphisms)  between local $\cin$-ringed
spaces we need to recall that (pre)sheaves can be pushed forward by
continuous maps.
Namely, if $\scS$ is a presheaf on a space $X$ and $f:X\to Y$ is a
continuous map,  then there is a presheaf $f_* \scS$ on $Y$ which is
defined by 
\[
  f_* \scS (U) := \scS (f\inv (U))
\]
for every open set $U\subset Y$. It is not hard to check that if
$\scS$ is a sheaf then the pushforward presheaf $f_*\scS$ is also a
sheaf.
\end{remark}
The following example is meant to motivate the definition of a map
between two $\cin$-ringed spaces.
\begin{example} Consider  a smooth map $f:M\to N$ 
between two manifolds.    It induces a map $f_\#: \cin(N) \to \cin(M)$
between $\cin$-rings; $f_\#$ is given by
\[
f_\# h := h\circ f
\]  
for all $h\in \cin(M)$; i.e., $f_\# $ is just a pullback map.    There
is also more:  for every open set $U \subset N$ we have a map of
$\cin$-rings 
\[
f_{\#, U} : \cin (U) \to \cin(f\inv (U)), \qquad f_\# (h) = h\circ f.
\]  
More precisely $\cin(U) = \cin_N(U)$ and $\cin(f\inv(U)) = \cin_M
(f\inv (U)) = (f_* \cin_M) (U)$.   Moreover, the maps
$\{f_{\#,  U} \}_{U\subset N}$ are compatible with restrictions, and therefore
  pullback by $f$ defines a map of sheaves
  \[
f_\#: \cin_N\to f_* \cin_M.
\]
\end{example}
\begin{definition}
A {\sf morphism} (or a {\sf map}) of local $\cin$-ringed spaces from $(X,\scO_X)$ to
$(Y,\scO_Y)$ is a pair $(f, f_\#)$ where $f: X\to Y$ is continuous and
$f_\#: \scO_Y\to f_* \scO_X$ is a map of sheaves of $\cin$-rings on the space $Y$.
\end{definition}
\begin{remark}
  Note that the map $(f,f_\#): (X,\scO_X) \to (Y,\scO_Y)$ of local
  $\cin$-ringed spaces induces, for every point $x\in X$,  a map of
  stalks $f_x: \scO_{Y,f(x)} \to \scO_{X, x}$
  
  In algebraic geometry one defines a morphism of locally ringed
  spaces to be a pair $(f,f_\#): (X,\scO_X) \to (Y,\scO_Y)$ so that
  the induced maps on stalks $f_x$ are all local: they are required to
  take the unique maximal ideal to the unique maximal ideal.     For
  local $\cin$-ringed spaces this is unnecessary because any map of
  local $\cin$-rings automatically preserves $\R$-points.
\end{remark}  

\begin{remark} \label{rmrk:2.46-7}
In order to define the composition of two maps of $\cin$-ringed spaces
we need to recall that pushing (pre)sheaves forward along a continuous
map $f:X\to Y$ is actually a functor from the category $\Pre_X$ of
presheaves on $X$ to the category $\Pre_Y$ of presheaves on $Y$: given
a map of presheaves $\varphi:\scO\to \scO'$ in $\Psh_X$ we get a map
of presheaves $f_*(\varphi): f_*\scO\to f_* \scO$.  Namely for an
open set $U\subset Y$ the $U$-component $\left(f_*(\varphi)\right)_U: f_*\scO(U)
\to f_*\scO' (U)$
of $f_*(\varphi)$ is, by definition, 
\[
\varphi_{f\inv (U)}: \scO(f\inv (U)) \to \scO'(f\inv (U)).
\]  
\end{remark}
\begin{definition}
Given two maps $(f, f_\#): (X, \scO_X) \to (Y, \scO_Y)$ and $(g,
g_\#): (Y, \scO_Y) \to (Z, \scO_Z)$ of local $\cin$-ringed spaces we define their composite to be
$(g\circ f, (g\circ f)_\#) :  (X, \scO_X) \to  (Z, \scO_Z)$, where
\[
(g\circ f)_\# : = (g_*(f_* \scO_x) \xleftarrow{g_* (f_\#)} g_* \scO_Y)
\circ (g_* \scO_Y \xleftarrow{g_\#} \scO_Z).
\]
Here we used  Remark~\ref{rmrk:2.46-7} and the fact that $(g\circ f)_* = g_* \circ f_*$ 
\end{definition}  

Strictly speaking the following definition is a proposition, which we won't prove.
\begin{definition}[The category $\LCRS$] \label{def:lcrs}
Local $\cin$-ringed spaces and their morphisms form a category which
we denote by $\LCRS$.
\end{definition}

\subsection{Differential spaces}
We next review the definition and some properties of a class of spaces that give rise to local $\cin$-ringed spaces.  
These are differential spaces in the sense of Sikorski \cite{Si}.  Traditionally the work on differential spaces avoids explicit mention of $\cin$-rings and of sheaves.
This has its advantages and its problems.

\begin{definition}[Differential space in the sense of Sikorski]\label{def:sikorski}
 A {\sf differential space} is a pair $(M, \scF)$ where $M$ is a
 topological space and $\scF$ is a (nonempty) set of real valued functions on $M$ so that
 \begin{enumerate}
   \item The topology on $M$ is the smallest topology making all
     functions in $\scF$ continuous; \label{def:sikorksi:it1}
 \item For any natural number $k$, for any $a_1,\ldots, a_k \in \scF$ and any
   smooth function $f\in \cin(\R^k)$ the composite $f\circ
   (a_1,\ldots, a_k)$ is in $\scF$; \label{closure}\label{def:sikorksi:it2}
\item If $g: M\to \R$ is a function with the property that  for any $x\in M$ there is an
  open neighborhood $U$ of $x$ and $f\in \scF$ ($f$ depends on $x$)
  so that $g|_U = f|_U$ then $g\in \scF$. \label{def:sikorksi:it3}
\end{enumerate}
The  whole set $\scF$ is called a  {\sf differential structure} on $M$.
\end{definition}

\begin{remark}\mbox{}
\noindent
  \begin{itemize}
 \item
One thinks of the family of functions $\scF$ on a differential
space $(M, \scF)$ as the set of all abstract smooth functions on
$M$. 
\item Condition (\ref{def:sikorksi:it2}) of
  Definition~\ref{def:sikorski} says that the smooth structure $\scF$ is a
  $\cin$-ring: for any $n>0$,
$a_1, \ldots, a_n\in \scF$ and $f\in \cin(\R^n)$ the composite
$f\circ (a_1,\ldots, a_n):M\to \R$ is again in $\scF$ by
(\ref{def:sikorksi:it2}).  Thus setting 
\[
f_{\scF} (a_1,\ldots, a_n) := f\circ (a_1,\ldots, a_n)
\]
for $f\in \cin(\R^n)$ and  $a_1, \ldots, a_n\in
\scF$ makes $\scF$ into a $\cin$-ring. \\[4pt]
(If $n=0$, then $\lambda \in \cin(\R^0) = \R$ is a constant and $\lambda_\scF:M\to \R$ is 
defined to be the constant function taking value $\lambda$ at all points of $M$.
Note that all constant functions have to be in $\scF$.  This is because  $\cin(\R^1)$ has constant functions, 
and  for any constant function $f\in \cin(\R)$ and any $a\in \scF$, the function $f_\scF(a) = f \circ a$ is constant.)

\item Condition  (\ref{def:sikorksi:it3}) 
  is a condition on the sheaf of
  functions on $M$ generated by $\scF$.  See Remark~\ref{rmrk:7.2}.
  It is often phrased as: ``if a function $g:M\to \R$ is locally
  smooth then it is smooth.''
\item Condition  (\ref{def:sikorksi:it1}) is equivalent to the existence of bump
  functions. See Appendix~\ref{app:bump}.
\end{itemize}
\end{remark}

\begin{definition}
Given a differential space $(M, \scF)$ we refer to the $\cin$-ring
$\scF$ as a {\sf differential structure} on the set $M$.

Note that by (\ref{def:sikorksi:it1}) the topology on $M$ is uniquely
determined by the differential structure $\scF$.
\end{definition}

\begin{example} \label{ex:mfld-diffspace} Let $M$ be a smooth manifold
  Then the pair $(M, \cin(M))$, where $\cin(M)$ is the set of
  $\cin$ functions, is a differential space.

  The only hard thing to check is that
  the topology on $M$ is the smallest among the topologies for which
  smooth functions are continuous.  But this follows from a theorem of
  Whitney: any closed subset of $M$ is the zero set of some smooth
  function (all our manifolds are Hausdorff and second
  countable).  See, for example, \cite[Theorem~2.29]{Lee}.  Hence if
  $U\subset M$ is open then there is a smooth function $f:M\to \R$ so
  that $U =\{x\in M \mid f(x) \not = 0\}$.
\end{example}

\begin{example}
  Let $M$ be a smooth manifold Then the pair $(M, C^0(M))$, where
  $C^0(M)$ is the set of continuous functions, is a differential
  space.
\end{example}

\begin{remark} \label{ex:ds_pd}
  If $(M, \scF)$ is a differential space then the $\cin$-ring $\scF$ is
  point-determined. This is because   a function on $M$ is non-zero only if there is a
  point at which it takes a non-zero value.
\end{remark}

Every differential space is a local $\cin$-ringed space.
Here are the details.

\begin{definition}[The structure sheaf $\scF_M$ of a differential space
  $(M, \scF)$] \label{def:str_sheaf}
Let  $(M, \scF)$ be a differential space.  Define a presheaf $\cP$
of $\cin$-rings on $M$ as follows: for an open set $U\subset M$ let 
\[
  \cP(U): = \{ f|_U \mid f\in \scF\}
\]
and let the structure maps of $\cP$  be the restrictions of functions.
Define the {\sf structure sheaf} $\scF_M$ of  the differential space $(M, \scF)$ to be the
sheafification of $\cP$.
\end{definition}
The structure sheaf $\scF_M$ has a concrete description: 

\begin{lemma} \label{lem:7.1} Let  $(M, \scF)$ be a differential
  space,  $\cP$ the presheaf on $M$ defined by 
 \[
   \cP(U): = \{ f|_U \mid f\in \scF\}
 \]
 as in Definition~\ref{def:str_sheaf}. 
  For an open subset
$U\subseteq M$ let 
\[
\begin{split}  
  \scF_M(U) :=& \{f:U\to \R\mid \textrm{ for any }x\in U \textrm{ there is
    an open neighborhood } V \\
  &\textrm{ of }x \textrm{ in }U
  \textrm{ and } g\in \cin(M) \textrm{ with } g|_V = f|_V\}.
\end{split} 
\]  
Define the restriction maps $\rho^U_V: \scF_M(U) \to \scF_M(V)$ to be the restrictions of functions: 
\[
\rho^U_V (h):= h|_V.
\]
 Then $\scF_M :=\{\scF_M(U)\}_{U\in \Open(M)}$ is a sheaf.  Moreover $\scF_M$ is a sheafification of the presheaf $\cP$. 
\end{lemma}

\begin{proof}[Proof of Lemma~\ref{lem:7.1}]
It is easy to see that the restriction maps $\rho^U_V: \scF_M(U) \to \scF_M(V)$ are well-defined. Consequently $\scF_M$ is a presheaf.  Moreover  $\cP(U) \subset
\scF_M (U)$ for all opens $U\in \Open(M)$.  Furthermore the presheaves $\scF_M$ and $\cP$ have the same
stalks.  So to prove that $\scF_M$ is a sheafification of $\cP$ it is enough to
check that $\scF_M$ is a sheaf.  This is not hard.
To finish the proof one checks that $\scF_M$ is a sheaf of
$\cin$-rings.
\end{proof}

\begin{remark} \label{rmrk:7.2}
We now see what condition~(\ref{def:sikorksi:it3}) of
  Definition~\ref{def:sikorski} is saying.  It says that  the  set  of functions
  $\scF$ on a differential space $(M, \scF)$ is the $\cin$-ring of
  global sections of the sheaf $\scF_M$ induced on $M$ by $\scF$
  and that  we don't get any more global sections.
\end{remark}

\begin{remark} \label{rmrk:2.57}
Note that for an open subset $U$ of $M$ a function $f$ is in
$\scF_M(U)$ if and only if there is an open cover $
\{U_\alpha\}_{\alpha\in A} $ of $U$ and a set $\{g_\alpha\}_{\alpha
  \in A}\subset \scF$ so that $f|_{U_\alpha} = g_\alpha|_{U_\alpha}$
for all $\alpha$.
\end{remark}  
\begin{lemma} \label{lem:2.58}
Let $(M, \scF)$ be a differential space and $\scF_M$ the induced sheaf
of $\cin$-rings.  Then the stalks of $\scF_M$ are local, hence the
pair $(M,
\scF_M)$ is a local $\cin$-ringed space.
\end{lemma}

\begin{proof}
Let $x\in M$ be a point and $\scF_{M,x}$ the stalk of $\scF_M$ at $x$.
Recall that $\scF_{M,x}$ is the set of equivalence classes of pairs $(U,
f)$ where $U$ is an open neighborhood of $x$ and $f\in \scF_M(U)$ a
smooth function.  In other words elements of $\scF_{M,x}$  are germs of functions.
We denote an element of $\scF_{M,x}$ either by
$[(U,f)]$ or by $f_x$, depending on whether we want to emphasize that
that the function representing the germ is defined on a neighborhood
$U$ or if we want to emphasize  that germ is at the point $x$ (and the
neighborhood $U$, on which $f$ is defined, is left unmentioned).

We have a well-defined evaluation map
\[
\ev_x: \scF_{M,x}\to \R, \qquad \ev_x (f_x) = f(x).
\]
Since the evaluation map is onto, its kernel $\fm$ is a maximal ideal
in $\scF_{M,x}$.    To prove that $\scF_{M,x}$ is local, it is enough to
show that any element of the complement of $\fm$ in $\scF_{M,x}$ is a
unit.

Suppose $f_x \in \left(\scF_{M,x}\right) \smallsetminus \fm$.  It is no loss of
generality to assume that $a:= \ev_x(f_x) = f(x) >0$.   Choose a smooth function
$\eta $ on $\R$ so that $\eta >0$ and $\eta(t) = t$ for all $t\in
(a/2, 3a/2)$.    Since $\eta > 0$, $\zeta (t) = 1/\eta(t)$ is
well-defined and smooth.  Consider $V= f\inv ((a/2, 3a/2))$.   For any
point $y\in V$,   $\eta (f (y)) = f(y)$.  On the other hand,
\[
1 =(\eta\cdot  \zeta) (f(y)) = \eta (f(y))\cdot  \zeta (f(y))
\]
for all $y\in V$.
It follows that $[(V, \zeta \circ f)]$ is the inverse of $f_x$ in the
stalk $\scF_{M, x}$.  Hence the germ $f_x$ is a unit, as we wanted to show.
\end{proof}

Differential spaces form a category that we will denote by $\DiffSp$.
The morphisms of this category are smooth maps that are defined as
follows.

\begin{definition}  A  {\sf smooth map} from a
  differential space  $(M,\scF)$  to a differential space $(N,
  \scG)$ is a function 
  $\varphi:M\to N$ so that for
  any $f\in \scG$ the composite $f\circ \varphi$ is in $\scF$.
\end{definition}

\begin{remark}
We have observed in Lemma~\ref{lem:2.58} above that to every differential space $(M,\scF)$ one
can associate a local $\cin$-ringed space $(M, \scF_M)$.  The map sending a differential space to the
corresponding local $\cin$-ringed space  can be upgraded to a fully faithful
functor $I:\DiffSp \to \LCRS$.  That is,  the category of differential
spaces is (isomorphic to) a full subcategory of the category $\LCRS$
of local $\cin$-ringed spaces.  See Theorem~\ref{thm:8.2} in Appendix~\ref{app:embed}.
\end{remark}

The rest of this subsection will be used for computing examples in the
next section.  It will be convenient at this point to recall the
notion of the initial topology.
\begin{definition}
Let $X$ be a set and $\scF$ a set of maps from $X$ to various
topological spaces.  The smallest topology on $X$ making all functions
in $\scF$ continuous is called {\sf initial}.

In particular a collection of real-valued functions $\scF$ on a set
$X$ uniquely defines an initial topology on $X$ (we give the real line
$\R$ the standard topology, of course).
\end{definition}  

\begin{definition}
A differential structure $\scF$ on a set $M$ is {\sf generated} by a subset
$A\subseteq \scF$ if $\scF$ is the smallest differential structure
containing the set $A$.  That is, if $\scG$ is a differential
structure on $M$ containing $A$, then $\scF\subseteq \scG$.
\end{definition}  

\begin{lemma}\label{rmrk:2.22}
Given a collection $A$ of
real-valued functions on a set $M$ there is a differential structure $\scF$
on $M$ generated by $A$.   %
The initial topology for $\scF$ 
is the initial topology for the set  $A$. %
\end{lemma}
\begin{proof}
See \cite[Theorem~2.1.7]{Sn}.
\end{proof}

\begin{notation}
We write $\scF =\langle A \rangle$ if the differential structure  $\scF$ is generated by the set $A$.
\end{notation}   
\begin{definition} \label{def:induced_diff_str}
Let $(M,\scF)$ be a differential space and $N\subseteq M$ a subset.
The {\sf  subspace differential structure } $\scF^N$ on $N$, also
known as the {\sf induced differential structure}, is the differential
structure on $N$ generated by the restrictions to $N$ of the
functions in $\scF$:
\[
\scF^N := \langle \scF|_N\rangle.
\]  
\end{definition}

\begin{lemma} \label{lem:2.24}
Let $(M,\scF)$ be a differential space and $(N, \scF^N)$ a subset of
$M$ with the induced subspace differential structure.     Then the
smallest 
topology on $N$  making all the functions of $\scF^N$ continuous
agrees with the subspace topology on $N$ coming from the inclusion
$i:N\hookrightarrow M$.
\end{lemma}  

\begin{proof}  The initial topology for the set $\scF|_N$ of
  generators of $\scF^N$ is the subspace topology.  Consequently the
  initial topology for $\scF^N = \langle \scF|_N\rangle$ is also the
  subspace topology (cf.\ Lemma~\ref{rmrk:2.22}).
\end{proof}

\begin{remark} \label{rmrk:2.66aug}
  The subspace differential structure $\scF^N$ can be
  given a fairly explicit description:
\begin{align*}
\scF^N= \{ f:N\to \R\mid & \textrm{ there is a collection of sets }
\{U_i\}_{i\in I}, \textrm{ open in M, with } \bigcup _i
U_i \supset N \\  & \textrm{ and a collection } \{g_i\}_{i\in I}
\subseteq \scF \textrm{ such that } f|_{N\cap U_i}  =
g_i|_{N\cap U_i } \textrm { for all indices } i\}.
\end{align*}
\end{remark}  

\begin{remark}\label{rmrk:cin U}
Let $(M,\scF)$ be a differential space and $(N, \scF^N)$ a subset of
$M$ with the induced subspace differential structure.   Then the
inclusion map $i:N\hookrightarrow M$ is smooth since for any $f\in
\scF$, $f\circ i = f|_N \in \scF^N$ by definition of $\scF^N$.

The subspace differential structure $\scF^N$ is the {\em smallest}
differential structure on $N$ making the inclusion $i:N\to M$ smooth.
This is because any differential structure $\scG$ on $N$ making the inclusion
$i:(N,\scG) \hookrightarrow (M, \scF)$ smooth must contain the set $\scF|_N$,
\end{remark}

\begin{lemma}  \label{lem:diff_subspaces_embed}
Let $(M,\scF)$ be a differential space and $(N, \scF^N)$ a subset of
$M$ with the induced subspace differential structure.  For any
differential space $(Y,\scG)$ and for any smooth map $\varphi:Y\to M$
that factors through the inclusion $i:N\to M$ (i.e., $\varphi(Y)
\subset N$), the map $\varphi: (Y,\scG)\to (N, \scF^N)$ is smooth.
\end{lemma}

\begin{proof}
We need to show that $\varphi^* h \equiv h\circ \varphi \in \scG$ for
any $h\in \scF^N$.     For any $f\in \scF$,
\[
\scG \ni f\circ \varphi = f\circ i \circ \varphi = \varphi^*(f|_N).
\]
Consequently $\varphi^* \left(\scF|_N\right) \subseteq \scG$.  Since $\scF|_N$
generates $\scF^N$ and since $\scG$ is a differential structure, we
must have $\varphi^* \left( \scF^N \right) \subseteq \scG$ as well.
\end{proof}
\begin{remark} \label{rmrk:cinU}
Let $(M, \scF)$ be a differential space, $\scF_M$ the corresponding
sheaf of $\cin$-rings and $U\subset M$ an open subset.  Then the
induced subspace  differential structure $\scF^U$ on $U$ is the value
of the sheaf $\scF_M$ on $U$:
\[
\scF^U =\scF_M(U).
\]  
This follows from Remark~\ref{rmrk:2.66aug} and from the definition of the
sheaf $\scF_M$.
\end{remark}

In the case where the differential space $(M, \scF)$ is a manifold and
$N$ is a subset of $M$, the subspace differential structure  $\scF^N$
has a simple description: 

\begin{lemma}\label{rmrk:2.13} \label{lem:2.27}
Let $M$ be a manifold and $N$ a subset of $M$, 
Then $f:N\to \R$ is in $\cin(M)^N:= \langle \cin(M)|_N\rangle$ (the subspace differential
structure on $N$) if and only if there is 
an open neighbourhood $U$ of $N$ in $M$ and a smooth function $g:U\to \R$ 
such that $f = g|_N$.

Moreover, if $N$ is closed in $M$, we may take $U= M$.  And then
$\cin(M)^N:=  \cin(M)|_N$. 
\end{lemma}

\begin{proof}%
Suppose $f\in \cin(M)^N$.  Then there is an collection of
open sets $\{U_i\}_{i\in I}$ with $N\subset \bigcup_i U_i$ and
$\{g_i\}_{i\in I}\subset \cin(M)$ so that $f|_{U_i \cap N} =
g_i|_{U_i\cap N}$ for all $i$ (cf.\  Remark~\ref{rmrk:2.66aug}).  Let $U = \bigcup_i U_i$.  There is a
partition of unity $\{\rho_i\}_{i\in I}$ on $U$ subordinate to the cover $\{U_i\}_{i\in I}$. 
Consider $g := \sum \rho_i g_i \in \cin(U)$.  Then $g|_N= f$.

If $N$ is closed,  then $\{U_i\}_ {i\in I} \cup \{M\smallsetminus N\}$
is an open cover of $M$.
Choose a partition of unity $\{\rho_i\}_{i\in I} \cup \{\rho_0\}$
subordinate to { this} cover of $M$ (with $\supp \rho_0\subset
M\smallsetminus N$). 
and again set $g := \sum_{i\in I} \rho_i g_i$.
Then $g$ is a smooth function on all of $M$, and $g|_N = f$.

Conversely suppose $U$ is an open subset of $M$ containing $N$.  Then
the inclusion $N\hookrightarrow M$ (which is smooth) factors through
the inclusion  $N\hookrightarrow   U$.  Since $\cin(U)$ is  the
subspace differential structure on $U$ (cf.\ Remark~\ref{rmrk:cinU})
Lemma~\ref{lem:diff_subspaces_embed} implies that 
the inclusion  $(N, \cin(M)^N) \hookrightarrow  (U, \cin(U))$ is smooth.
In pariticular, for any $g\in \cin(U)$ the restriction $g|_N$ is in $\cin(M)^N$.
\end{proof}

\begin{remark}
Lemma~\ref{lem:2.27} holds in greater generality.  The proof does not
really use the fact that $M$ is a manifold, it only needs the
existence of the partition of unity.   Partitions of unity do exist for second
countable Hausdorff locally compact differential spaces; see \cite{Sn}.
\end{remark}

\subsection{$\cin$-schemes}

We next discuss $\cin$-schemes.  The definition is 
is due to Dubuc
\cite{Dubuc}.  Dubuc used $\cin$-schemes to construct a model of
synthetic differential geometry.   Algebraic geometry over $\cin$-rings  has been developed further by Joyce \cite{Joy}.  

There is an evident global sections functor $\Gamma: \LCRS \to \op{\cring}$
from the category of local $\cin$-ringed spaces to the category opposite  of the category
$\cin$-rings.  It is given  by 
\[
\Gamma\left( (N, \scB)) \xrightarrow{(f,f_\#)} (M,\scC) \right) := 
\left(\scC(M)\xrightarrow{f_\#} (f_*\scB) (M) = \scB(f\inv (M)) = \scB (N)\right).
\]
Thanks to a
theorem of Dubuc \cite{Dubuc}, the global section functor $\Gamma$ has
a right adjoint
\[
  \Spec: \op{\cring} \to \LCRS ,
\]
the {\sf spectrum functor}.  One way to construct $\Spec$ is to proceed in
the same way as in algebraic geometry.  Namely, given a $\cin$-ring
$\scC$ one considers the set $\cX_\scC$ of all $\R$-points of $\scC$.
Then $\cX_\scC$ is given a topology and finally one constructs a sheaf
$\scO_\scC$ of local $\cin$-rings on the space $\cX_\scC$.  See
\cite{Joy} for details.

\begin{definition} An affine {\sf $\cin$-scheme} is a local
  $\cin$-ringed space isomorphic to $\Spec(\scC)$ for some $\cin$-ring
  $\scC$.

The {\sf category of affine $\cin$-schemes} $\Aff$ is the essential
image of the functor $\Spec$.
\end{definition}

\begin{remark}
For any manifold $M$ the local $\cin$-ringed space $(M,\cin_M)$ is isomorphic to the affine scheme $\Spec(\cin(M))$ (see \cite{Dubuc} and \cite{Joy}).
In particular the real and complex projective spaces $\R P^n$ and $\C
P^n$ are affine $\cin$-schemes.   For this reason the need for {\em
projective } $\cin$-schemes is not obvious.
\end{remark}

\begin{remark}
  Unlike the spectrum functor in algebraic geometry
  $\Spec: \op{\cring} \to \Aff$ is not an equivalence of categories.
  For example $\Spec$ assigns the empty space to any $\cin$-ring with
 no $\R$-points, such as
  $\cin(\R^n)/\cin_c (\R^n)$ (here as before $\cin_c(\R^n)$ is the ideal of
  compactly supported functions) or the zero $\cin$-ring 0.
\end{remark}

\begin{remark}  The relation between $\cin$-schemes and differential
  spaces is a bit messy.   For a $\cin$-scheme to be a differential
  space its structure sheaf must be a sheaf of functions.  In
  particular it can have no nilpotents.  As as result
  many (most?) $\cin$-schemes are not differential spaces.

  On the other hand, given a differential space $(M,\scF)$ any point
  $x\in M$ gives rise to the $\R$-point $ev_x:\scF\to \R$, the
  evaluation at $x$.  However there are differential spaces where the
  set of $\R$-points is bigger than their  set of (ordinary) points,
  see \cite{CSt}.

  Worse yet, a differential space may have too many $\R$-points. Here
  is an extreme example.  Take any set $M$. Give it the  trivial topology.
  The set $\scF$ of constant $\R$-valued functions on $M$ is a
  differential structure.  As a $\cin$-ring $\scF$ is isomorphic to
  $\R$ and consequently has only one $\R$-point.

In other words,  while the topology on the differential space $(M,
  \scF)$ is uniquely determined by its differential structure $\scF$,
  the set $M$ itself is not uniquely determined by $\scF$.  Compare
  this with  affine $\cin$-schemes,  where the set of
  points of a local $\cin$-ringed space $\Spec(\scC)$ is uniquely determined
  by the $\cin$-ring $\scC$.  
\end{remark}

\subsection{Graded objects}
\begin{remark}
There are two conventions in the literature regarding ($\N$- or $\Z$-)
graded objects in an abelian category: they can be viewed either as
direct sums or as sequences of objects. Kamps and Porter
\cite[p.~207]{KP}, for instance, use the second convention.

Depending on the category one is working in the choice of the
convention may or may not matter. For example, a $\Z$-graded real
Lie algebra $\fg$ can be viewed as a direct sum
$\fg=\oplus_{n\in \Z} \, \fg_n$ together with an $\R$-bilinear map
(the bracket) $\fg\times \fg\to \fg$ satisfying the appropriate
conditions.  It can equivalently be viewed as a sequence of vector
spaces $\{\fg_n\}_{n\in \Z}$ together with a family of linear maps
$\{\fg_n \otimes \fg_m \to \fg_{n+m}\}_{n,m\in \Z}$ (again, subject to
conditions).  So for graded
Lie algebras the two conventions are equivalent.

However, in the case of graded sheaves of modules, the two approaches are
not equivalent: see  Remark~\ref{rmrk:graded_sheaves}.
Since we will need the (pre-)sheaf versions of various graded objects,
we use the second convention throughout: a graded object is a sequence
of objects.  Since graded derivations of non-negatively graded algebras may have negative degrees, 
our graded objects are all graded by the integers $\Z$.  Note that any algebra 
graded by the natural numbers $\N$ can be viewed as a $\Z$-graded
object by placing zeros in negative degrees.
\end{remark}

The following definition is a variation on  the conventional definition of
a commutative graded algebra.   %
\begin{definition} A {\sf commutative graded algebra} (CGA) over a
  $\cin$-ring $\scC$ is a sequence of $\scC$-modules $\{B^k\}_{k\in \Z}$
  together with a sequence of $\scC$-bilinear maps $\wedge: B^k\times
  B^\ell \to B^{k+\ell} $ (or, equivalently of $\scC$-linear maps   $\wedge: B^k\otimes_{\scC}
  B^\ell \to B^{k+\ell} $; we don't notational distinguish between
  the two) so that $\wedge $ is associative and graded
  commutative. That is for all $k,\ell, m\in \Z$ and all $a\in B^k$,
  $b\in B^\ell$ and $c\in B^m$
  \begin{enumerate}
  \item $(a\wedge b) \wedge c = a\wedge (b\wedge c)$ and
  \item $b\wedge a = (-1)^{k\ell} a \wedge b$ (Koszul sign convention).
  \end{enumerate}
  We denote such a CGA by $(B^\bullet, \wedge)$. We write $b\in
  B^\bullet$ if $b\in B^k$ for some $k\in \Z$.
  We refer to elements of $B^k$ as {\em homogeneous elements of degree
    k} and write $|b| =k$ to indicate that $b\in B^k$.
\end{definition}

\begin{example} \label{ex:exterior}
Let $\scC$ be a $\cin$-ring and  $\scM$ an $\scA$-module. Then the
exterior algebra $\Lambda ^\bullet \scM =\{ \Lambda^k \scM\}_{k\in \Z}$
  together with the usual wedge product $\wedge$ 
  is a CGA over $\scC$.  Note that the exterior powers are taken over
  $\scC$. Note also that $\Lambda^0 \scM = \scA$ and that
  $\Lambda^k \scM=0$ for $k<0$ by convention.
\end{example}  
\begin{remark}
Commutative graded algebras (CGAs) are really graded-commutative algebras:
that is, they are graded and they are commutative up to the appropriate sign;
 they are not commutative on the nose.    This is why sometimes they are also 
called GCAs.
\end{remark}

\begin{definition}[The category $\CGA$ of commutative graded algebras] \label{def:catCGA}
Recall that modules over $\cin$-rings form a category $\Mod$ (Definition~\ref{def:mod}).
Consequently commutative graded algebras over $\cin$-rings form a
category as well.  We denote it by $\CGA$.  The objects of this category are pairs $(\scC,
(\scM^\bullet, \wedge))$,  where $\scC$ is a $\cin$-ring and
$(\scM^\bullet, \wedge)$ is a CGA over $\scC$.  A morphism from $(\scC_1,
(\scM_1^\bullet, \wedge))$ to $(\scC_2,
(\scM_2^\bullet, \wedge))$ in $\CGA$ is a pair $(\varphi, f
= \{f^k\}_{k\in \Z})$
where, for each $k$,  $(\varphi, f^k): (\scC_1, \scM_1^k)\to (\scC_2, \scM_2^k)$
is a morphism in the category $\Mod$ of modules and, additionally, $f$ preserves the
multiplication: $f(m\wedge m') = f(m)\wedge f(m')$.
\end{definition}

\begin{definition} \label{def:gr-der}
Let $(B^\bullet, \wedge)$ be a commutative graded
algebra over a $\cin$-ring $\scC$.  A {\sf graded derivation of degree
  $k\in \Z$} is a sequence of  $\R$-linear
maps 
$X= \{X^\ell :B^\ell \to B^{\ell+k}\}_{\ell\in \Z}$  so that for all homogeneous
elements $x,y\in B^\bullet$
\[
X(x\wedge y) = X(x) \wedge y + (-1)^{|x|k} x \wedge X(y),
\]
where $|x|$ is the degree of $x$, i.e., $x\in B^{|x|}$ (to reduced the
clutter the $X$ on the left stands for $X^{|x|+|y|}$ and so on). As
always our signs follow the Koszul sign convention.
\end{definition}

\begin{example}
Let $M$ be a smooth manifold, $\scC = \cin (M)$, the $\cin$-ring of
smooth functions.  Then the exterior derivative
$d:\bOmega_{dR}^\bullet(M)\to \bOmega_{dR}^{\bullet +1}(M)$ is a graded derivation
of degree +1 of the CGA $(\bOmega_{dR}^\bullet (M), \wedge)$ of global de
Rham differential forms.

Any vector field $v:M\to TM$ give rise to a derivation $\imath_v:
\bOmega_{dR}^\bullet(M)\to \bOmega_{dR}^{\bullet -1}(M)$ of degree -1:
it is the
contraction with $v$. The ordinary Lie derivative $\cL_v: \bOmega_{dR}^\bullet(M)\to
\bOmega_{dR}^{\bullet}(M)\,$ with respect to the vector field $v$ is a
degree 0 derivation of $(\bOmega_{dR}^{\bullet}(M), \wedge)$.

\end{example}

The following lemma will be useful when working with contractions.
\begin{lemma} \label{lem:1level0}
  Let $(\scM^\bullet, \wedge)$ be a CGA over a $\cin$-ring $\scC$ with
  $\scM^k = 0$ for all $k<0$ and with $\scM^0 = \scC$. Any degree -1
  derivation $X:\scM^\bullet \to \scM^{\bullet-1}$  is $\scC$-linear.
\end{lemma}  
\begin{proof}
For all $a\in \scC =\scM^0$ and all $m\in
  \scM^\bullet$
\[
X(a\wedge m) = X(a) \wedge m + (-1)^{0\deg (m)} a \wedge X(m) = a
\wedge X(m)
\]
since $X(a) \in \scM^{-1}= 0$.
\end{proof}

\begin{remark}
  The set of all graded derivations of degree $k$ of a CGA
  $(B^\bullet, \wedge)$ over a $\cin$-ring $\scC$ is a vector space
  over $\R$ and an $\scC$-module.  The operations are defined
  pointwise.  For example, if $X=\{X^\ell\}_{\ell \in \Z} $ is a
  derivation and $a\in \scC$, $b\in B^\bullet $ then
  $(aX) (b) := a(X (b)) $ and so on.
\end{remark}

\begin{definition}
  A commutative {\sf differential } graded algebra (CDGA) over a $\cin$-ring
  $\scC$ is a commutative graded algebra $(\scB^\bullet, \wedge)$
  together with a degree +1 derivation $d \in \Der^1(\scB^\bullet)$
  called a {\sf differential}, which squares to 0. That is, 
  $d\circ d =0$.

  We write $(\scB^\bullet, \wedge, d)$ to denote such a CDGA.
\end{definition}

\begin{definition}[The category $\CDGA$ of commutative differential
  graded algebras] \label{def:catCDGA}  The objects of the category
  $\CDGA$ of commutative differential graded algebras are  triples 
  $(\scC, (\scM^\bullet, \wedge), d)$, where $\scC$ is a $\cin$-ring,
  $(\scM^\bullet, \wedge)$ is a CGA over $\scC$ and $d:\scM^\bullet
  \to \scM^\bullet$ is a degree 1 derivation of   $(\scM^\bullet,
  \wedge)$.  Equivalently the objects of $\CDGA$ are pairs   $(\scC,
  (\scM^\bullet, \wedge, d))$, where $\scC$ is a $\cin$-ring and $
  (\scM^\bullet, \wedge, d))$ is a CDGA over $\scC$.

  A morphism from  $(\scC_1, (\scM_1^\bullet, \wedge), d)$ to
  $(\scC_2, (\scM_2^\bullet, \wedge), d))$ is a morphism $(\varphi,
  f): (\scC_1, (\scM_1^\bullet, \wedge)) \to (\scC_2, (\scM_2^\bullet,
  \wedge))$ of commutative graded algebras with the additional
  property that $f$ commutes with the differentials.  
\end{definition}

We now turn out attention to $\Z$-graded presheaves and sheaves.
\begin{definition} \label{def:graded_ob}
  A {\sf $\Z$-graded  presheaf $\scP^\bullet$} of
$\scA$-modules over a local $\cin$-ringed space
$(M,\scA)$ is a sequence $\{\scP^n\}_{n\in \Z}$ of presheaves of
$\scA$-modules.  Similarly, a 
{\sf $\Z$-graded sheaf $\scS^\bullet $} of
$\scA$-modules over a local $\cin$-ringed 
$(M,\scA)$ is a sequence $\{\scS^n\}_{n\in \Z}$ of sheaves of
$\scA$-modules.
\end{definition}

\begin{remark} \label{rmrk:graded_sheaves}
Since the category of modules over a $\cin$-ring has coproducts so
does the category of presheaves of modules over a fixed local
$\cin$-ringed space $(M, \scA)$.  Moreover coproducts of presheaves 
are computed object-wise.  Consequently given a sequence 
$\{\scS^n\}_{n\in \Z}$ of presheaves of $\scA$-modules their coproduct
$ \oplus_{n\in \Z} \scS_n$ is defined by 
\[
\left( \oplus_{n\in \Z} \scS^n \right) (U)= \oplus_{n\in \Z} \scS^n(U),
\]
for all open sets $U\subset M$.  Here on the right the direct sum is
taken in the category of $\scA(U)$-modules.  Since the sheafification
functor $\sh:\Psh_M\to \Sh_M$ from presheaves to sheaves is left
adjoint to inclusion functor $i:\Sh_M\hookrightarrow \Psh_M$, the
functor $\sh$ preserves colimits and, in particular, coproducts.  Therefore if
$\{\scS^n\}_{n\in \Z}$ is a sequence of sheaves, then their coproduct (``direct sum'')
in $\Sh_M$ exists and equals $\sh (\oplus_{n\in \Z} \scS^n )$.  In
particular for an open set $U\subset M$, the $\scA(U)$-module
$\sh (\oplus_{n\in \Z} \scS^n ) (U)$ is, in general, different from
$\oplus_{n\in \Z} \scS^n (U)$. See Example~\ref{ex:pre-direct_sum}
below.  Of course $\sh (\oplus_{n\in \Z} \scS^n )$ (with the
appropriate structure maps) {\em is} a coproduct in $\Sh_M$, so one
can view $\Z$-graded objects in $\Sh_M$ as $\Z$-indexed coproducts.
One can even make sense of a map
$f: \sh (\oplus_{n\in \Z} \scS^n ) \to \sh (\oplus_{n\in \Z} \scS'^n
)$ as having a  degree $k$: we require that there are maps
$f^m:\scS^m\to {\scS'}^{m+k}$ of sheaves of modules making the diagrams
\[
\xy
(-15,10)*+{\sh (\oplus_{n\in \Z} \scS^n ) }="1";
(15,10)*+{\sh (\oplus_{n\in \Z} {\scS'}^n) }="2";
 (-15,-5)*+{\scS^m}="3";
(15,-5)*+{{\scS'}^{m+k}}="4";
{\ar@{->}^{f} "1";"2"};
{\ar@{->}^{\imath^m} "3";"1"};
{\ar@{->}_{{\imath'}^{m+k}} "4";"2"};
{\ar@{->}^{f^m} "3";"4"};
\endxy
\]
commute for all $m$ (here $\{\imath^m:\scS^m\to \sh (\oplus_{n\in \Z}
\scS^n ) \}_{m\in\Z}$  and $\{{\imath'}^m:{\scS'}^m\to \sh (\oplus_{n\in \Z}
{\scS'}^n ) \}_{m\in\Z}$ are the structure maps of the coproducts).
However it is not clear what the value of
the sheaf $\sh (\oplus_{n\in \Z} \scS^n )$ on some open subset
$U\subset M$ actually is in concrete situations.  This is why we use
Definition~\ref{def:graded_ob}.
\end{remark}

\begin{example}\label{ex:pre-direct_sum}  We  construct an example
  of a $\Z$-indexed collection of sheaves on a manifold so that their
  direct sum (as a presheaf) is not a sheaf.

  Let $B^n$ denote the open $n$-ball in $\R^n$ of radius 1 centered at
  0.  Let $M= \bigsqcup_{n\geq 0} B^n$, the coproduct of the balls $B^n$'s.  The
  space $M$ is a second countable Hausdorff manifold of unbounded
  dimension.  Let $\bOmega^k_{dR,M}$ denote the sheaf of ordinary
  differential $k$-forms on $M$ (with $\bOmega^k_{dR,M} \equiv 0$ for $k<0$).
Form the  direct sum (in the category of presheaves)
\[
\bOmega^\bullet_{dR,M}:= \oplus _{k\in \Z}\bOmega^k_{dR,M}.
\]
This presheaf {\it is not a sheaf}.  Here is one way to see it: for
each $n$ pick a volume form $\alpha_n\in \bOmega_{dR}^n (B^n)$.  The
collection of open balls $\{B^n\}_{n\geq 0}$ is an open cover of $M$
and all the elements of the cover are mutually disjoint.  If
$\oplus _{k\in \Z }\bOmega^k_{dR,M}$ were a sheaf, we would have an
element $\alpha\in \oplus _{k\geq 0 }\bOmega^k_{dR,M}$ so that
$\alpha|_{B^n} = \alpha_n$ for all $n$.  But any element of the direct
sum has a bounded degree, so no such $\alpha$ can exist.  Therefore
the direct sum $\oplus _{k\geq 0 }\bOmega^k_{dR,M}$ of sheaves 
(computed in the category $\Psh_M$ of presheaves!) 
 is not a sheaf.
\end{example}  

\begin{definition} \label{def:12.16} A {\sf presheaf of commutative
    graded algebras (CGAs)} $ (\scM^\bullet, \wedge)$ over a
  $\cin$-ringed space $(M, \scA)$ is a sequence of $\scA$-modules
  $\{\scM^k\}_{k\in \Z}$ together with a sequence of maps of
  presheaves $\wedge: \scM^k\otimes \scM^\ell \to \scM ^{k+\ell} $ of
  $\scA$-modules so that $\wedge $ is associative and graded
  commutative. That is, 
 we require that the appropriate diagrams in the category
of presheaves of $\scA$-modules commute.

Equivalently a presheaf $ (\scM^\bullet, \wedge)$ of CGAs over $(M, \scA)$ is a functor
\[
\op {\Open(M)}\to \CGA
\]
That is, for all open sets $U\subset M$ the pair
$ ( \{\scM^k(U)\}_{k\in \Z}, \wedge_U) $ is CGA over the $\cin$-ring
$\scA(U)$ and the restriction maps
$r^U_V: ( \scM^\bullet(U), \wedge_U) \to ( \scM^\bullet(V), \wedge_V)
$ are maps of CGAs (cf. Definition~\ref{def:catCGA}).

\end{definition}

\begin{definition} \label{def:12.16-1} A {\sf sheaf of commutative
    graded algebras (CGAs)} $ (\scM^\bullet, \wedge)$ over a
  $\cin$-ringed space $(M, \scA)$ is a presheaf of CGAs $
  (\scM^\bullet, \wedge)$  with the property that $\scM^k$ is a sheaf
  for each $k\in \Z$.
\end{definition}

\begin{definition} \label{def:gr-der_presh}
A {\sf graded derivation $X$ of degree $k$} of a presheaf $(\scM^\bullet
=\{\scM^k\}_{k\in \Z},
\wedge)$ of CGAs over a
  $\cin$-ringed space $(M, \scA)$  is a collection of maps of presheaves $\{X^\ell:
\scM^\ell \to \scM^{\ell +k}\}_{\ell \in \Z}$ of $\scA$-modules so that
\[
X_U: \scM^\bullet(U) \to \scM^{\bullet +k}(U)
\]
is a degree $k$ derivation of the CGA $(\scM^\bullet (U),
\wedge)$ of $\scA(U)$-modules
for each open
set $U\in \Open(M)$.     Equivalently $X =\{X^\ell\}_{\ell\in \Z}$ makes the appropriate diagrams
in the category of presheaves of $\scA$-modules commute.
\end{definition}

\begin{definition}
A {\sf presheaf of commutative differential graded algebras (CDGAs)} $
(\scM^\bullet, \wedge, d)$ over a $\cin$-ringed space $(M,\scA)$ is a
presheaf $ (\scM^\bullet, \wedge)$ of CGAs over $(M,\scA)$ together
with a sequence of maps $\{d:\scM^k \to \scM^{k+1}
\}_{k\in \Z}$ of presheaves of real vector spaces so that for each
open set $U\subset M$ the triple $(\scM^\bullet(U), \wedge_U, d_U)$ is
a CDGA over the $\cin$-ring $\scA(U)$.

Equivalently a presheaf CDGAs over a $\cin$-ringed space $(M, \scA)$
is a functor from the category $\op{\Open(M)}$ of open subsets of $M$
to the category $\CDGA$ of commutative differential graded algebras
(cf.\ Definition~\ref{def:catCDGA} ).
\end{definition}

\begin{definition}
A {\sf sheaf of CDGAs} is a presheaf of CDGAs
$(\scS^\bullet, \wedge, d)$ so that $\scS^k$ is a sheaf for all
$k\in \Z$.
\end{definition}

\section{K\"ahler differentials of $\cin$-rings} \label{sec:1form} In
this section we recall the construction of the $\cin$-ring analogue of
K\"ahler differentials of commutative algebra.

\begin{definition}
 A {\sf module of $\cin$-K\"ahler
differentials} over a $\cin$-ring $\scC$ is an $\scC$-module $\Omega^1_\scC$
together with a $\cin$-derivation $d_\scC:\scC\to
\Omega^1_\scC$, called the {\sf universal derivation}, with the following  property:  for any $\scC$-module $\scN$ and any derivation
$X:\scC\to \scN$ there exists a unique map of modules $\varphi_X:
\Omega^1_A\to \scN$ with
\[
\varphi_X \circ d_\scC = X.
 \] 
\end{definition}

\begin{remark}
  The module of $\cin$-K\"ahler differentials is not be confused with
  the module of ordinary algebraic K\"ahler differentials, since the
  two modules are quite different.  For example G\'omez \cite{Gomez}
  shows that the cardinality of a set of generators of the module of algebraic K\"ahler
  differentials $\Omega^1_{\cin(\R^n), alg}$ of $\cin(\R^n)$ is at
  least the cardinality of the reals.  Compare this with the fact that
  the module of $\cin$-K\"haler differentials $\Omega^1_{\cin(\R^n)}$
  is isomorphic, as a $\cin(\R^n)$-module to the module of de Rham
  1-forms $\bOmega^1 _{dR}(\R^n)$ --- see Example~\ref{ex:Krn}, which
  has rank $n$.  Note
  that $\Omega^1_{\cin(\R^n)} \simeq \bOmega^1 _{dR}(\R^n)$ is a free
  $\cin(\R^n)$-module.

\end{remark}

\begin{remark} \label{rmrk:3.3}
Let $\scC$ be a $\cin$-ring.
The universal property of the derivation $d:\scC\to \Omega^1_\scC$ says
that for any $\scC$-module $\scM$, the map of modules
\[
  d^*:\Hom (\Omega^1_\scC, \scM)\to \cDer(\scC, \scM),\quad
  d^*(\varphi): = \varphi \circ d 
\]
is an isomorphism (of $\scC$-modules).

Consider now the special case of $\scM = \scC$.  Denote the inverse of
$d^*$ by $\imath$.  In this notation, for any derivation $v\in
\cDer(\scC)$ there is a unique map of $\scC$-module
\[
\imath_v : \Omega^1_\scC\to \scC
\]  
so that
\[
\imath_v ( d_\scC c  ) = v(c)
\]
for all $c\in \scC$.  We think of $\imath_v \alpha$ as the {\em
  contraction} of a derivation $v$ and a differential $\alpha \in
\Omega^1_\scC$.  We then have a canonical map
\[
\psi: \Omega^1_\scC \to \Hom (\cDer(\scC), \scC)
\]  
with
\[
\psi(\alpha) v  = \imath _v \alpha 
\]  
for all $v$ and $\alpha$.   In the case where $M$ is a manifold and
$\scC = \cin(M)$, the space of derivations $\cDer(\cin(M)$ is the
space of vector fields on $M$ and can be identified with the space
$\Gamma(TM)$ of sections of the tangent bundle.  On the other hand,
the space $\Omega^1_{\cin(M)}$ of $\cin$-K\"ahler differentials  is
the space of ordinary 1-forms $\bOmega^1_{dR} (M) =\Gamma(T^*M)$ (cf.\
Remark~\ref{rmrk:5.12} below) and the map $\psi$ is then the map
\[
\psi: \Gamma (T^*M) \to \Hom (\Gamma(TM), \cin(M)), \quad \psi
(\alpha) (v) = \imath_v \alpha,
\]  
which is an isomorphism of $\cin(M)$-modules.  This is because
for a vector bundle $E\to M$ there is a canonical isomorphism
$\Gamma(E^*) \to \Hom_{\cin(M)} (\Gamma(E), \cin(M))$ of
$\cin(M)$-modules (here as before $\Gamma(E^*)$, $\Gamma(E)$ are the modules of global
sections of $E^*$ and  $E$, respectively).  See \cite[Proposition~7.5.4]{Conlon}.

In general the map $\psi $ is not an isomorphism  See
Example~\ref{ex:xy=0} of a differential space where $\psi$ is not
injective.   For this reason I think one should not define the space
of 1-forms on a 
on differential space $(M, \scF)$ to be $\Hom(\cDer(\scF), \scF)$.
\end{remark}  

We'll prove existence of K\"ahler differentials in
Theorem~\ref{thm:DK} below.  The result is  known --- it was proved
by Dubuc and Kock \cite{DK} over 30 years ago.   The proof we will
present is essentially the one in \cite{Joy}.   Our excuses for
providing the proof are that we (1) are trying to keep the paper
self-contained and (2)  we will need the details of the proof for the
construction of the  $\cin$-algebraic de Rham complex of a $\cin$-ring. The
uniqueness of K\"ahler differentials  follows easily from their
universal properties --- see Remark~\ref{rmrk:Kd:unique}.

We will compute two examples before proving Theorem~\ref{thm:DK}.   But
first, a few observations and remarks.

\begin{lemma} \label{lem:B.2}
Let $\scA$ be a $\cin$-ring, $w:\scA\to \scA$ a derivation and $a, b\in
\scA$ two elements with $ab=1$ (i.e., $a$ is invertible and $b$ is the
inverse of $a$).  Then
\begin{equation} \label{eq:B2}
w(b) = -b^2 w(a).
\end{equation} 
\end{lemma}

\begin{proof}
Since $w$ is a derivation, $w(1) = 0$.  Hence $0 = w(ab) = w(a) b+ a\, w(b)$ 
and the result follows.
\end{proof}

\begin{remark} \label{rmrk:Kd:unique}
It follows from the definition of the module $\Omega^1_\scC$ of K\"ahler differentials
of a $\cin$-ring $\scC$  that the module has to be unique in the following
sense: if $\scN$ is another $\scC$ module and $d:\scC\to \scN$ is another derivation with the same
universal property as $d_\scC:\scC\to \Omega_\scC^1$ then there is a
unique isomorphism $\psi: \scN\to \Omega_\scC^1$ of $\scC$-modules so that
\[
   \psi \circ d = d_\scC.
 \]
 Following a common practice we may refer to {\em any} $\scC$-module
 $\scN$ with a derivation $d:\scC \to \scN$ that has the
 appropriate universal property as ``the'' module of K\"ahler
 differentials of the $\cin$-ring $\scC$.
\end{remark}

\begin{example}  \label{ex:Krn}
The module $\bOmega_{dR}^1(\R^n)$ of ordinary differential 1-forms together with the exterior
derivative $d:\cin(\R^n) \to \bOmega_{dR}^1(\R^n)$ is a module of K\" ahler
differentials over the $\cin$-ring $\cin(\R^n)$:
$\Omega^1_{\cin(\R^n)} = \bOmega_{dR}^1(\R^n)$.

\begin{proof}   We need to show that for any $\cin(\R^n)$-module
  $\scN$ and for any derivation $X: \cin(\R^n) \to \scN$ there is a
  unique map of modules $\varphi_X: \bOmega_{dR}^1(\R^n) \to \scN$ so that
  \[
\varphi_X (dg) = X(g)
\]
for all $g\in \cin(\R^n)$.   Since $\bOmega_{dR}^1(\R^n)$ is freely
generated, as a module, by the differentials $dx_1, \ldots, dx_n$
(where $x_1,\ldots,x_n:\R^n\to \R$ are the coordinate functions),
the map $\varphi_X$ given by
\[
\varphi_X (\sum a_idx_i) := \sum a_i X(x_i)
\]
is a well-defined map of $\cin(\R^n)$ modules.  And it is the unique map
from $\bOmega_{dR}^1(\R^n)$ to $\scN$ satisfying $\varphi_X(dx_i) = X(x_i)$.

It is not hard to check that the exterior derivative $d:\cin(\R^n) \to
\bOmega_{dR}^1 (\R^n)$ is a $\cin$-ring derivation: any $k>0$, any
$h\in \cin(\R^k)$ and any $f_1,\ldots, f_k\in \cin(\R^n)$
\[
 \begin{split} 
d\left(h_{\cin(\R^n)} (f_1,\ldots, f_k)\right) &= d (h\circ (f_1,\ldots, f_k))\\
&= \sum \partial_i(
h(f_1,\ldots, f_k))\, dx_i = \sum_i \left(\sum_j (\partial_j h)
  (f_1,\ldots, f_k) \cdot \partial_i f_j   \right)\, dx_i\\
&=\sum_j\left( (\partial_j h) (f_1,\ldots, f_k) \cdot \sum_i \partial_i
  f_j \,dx_i\right)
= \sum _j (\partial_j h)_{\cin(\R^n)} (f_1,\ldots, f_k) \cdot df_j.
\end{split}
\]
It remains to check that $\varphi_X(dg) = X(g)$ for all $g\in \cin(\R^n)$.
The $\cin$-ring $\cin(\R^n)$ is a free ring on the generators
$x_1,\ldots, x_n$ (see Definition~\ref{def:free} and Lemma~\ref{lem:rn-is-free}).
In particular $g_{\cin(\R^n)} (x_1,\ldots, x_n) = g$ for all
$g\in \cin(\R^n)$. Since $X$ is a derivation
\[
\begin{split}
X(g) = X(g_{\cin(\R^n)} (x_1,\ldots, x_n)) &= \sum (\partial_i
g)_{\cin(\R^n)}(x_1,\ldots, x_n) \, X(x_i) \\&= \sum (\partial_i
g)\, X(x_i) = \varphi_X (\sum (\partial_i g)\, dx_i) = \varphi_X (dg).
\end{split}
\]
Thus there is a unique map of $\cin(\R^n)$-modules
$\varphi_X: \bOmega_{dR}^1(\R^n) \to \scN$ which satisfies
$\varphi_X \circ d = X$.  It is given by
$\varphi_X (\sum a_i dx_i) = \sum a_i X(x_i)$.  
\end{proof}
\end{example}

\begin{example}\label{ex:kropen}
Let $U\subset \R^n$ be an open set.  The module $\bOmega_{dR}^1(U)$ of
1-forms together with the exterior derivative $d:\cin(U) \to \bOmega_{dR}^1(U)$ is a module of K\"ahler differentials
  over the $\cin$-ring $\cin(U)$. That is, $\Omega^1_{\cin(U)} = \bOmega_{dR}^1(U)$.
\begin{proof}
  As in the case of $\R^n$, the differentials
  $\{ d(x_i|_U) =dx_i|_U\}_{1\leq i \leq n}$ of the standard coordinate
  functions $\{x_i|_U:U\to \R^n\}_{1\leq i \leq n}$ freely generate the $\cin(U)$-module
  $\bOmega_{dR}^1(U)$ of 1-forms.  Therefore for any
  $\cin(U)$-module $\scN$ and any derivation $X:\cin(U) \to \scN$
  there is a $\cin(U)$-module map
  $\varphi_X: \bOmega_{dR}^1 (U) \to \scN$ defined uniquely by
\[
\varphi_X \left(\sum a_i\, dx_i|_U\right) = \sum a_i \,X(x_i|_U).
\]
Using the fact that
\[
df = \sum (\partial _i f)\, dx_i|_U
\]
for all $f\in \cin(U)$ it is not hard to show that the exterior
derivative $d:\cin(U) \to \bOmega_{dR}^1(U)$ is a $\cin$-derivation.    It remains
to check that
\[
\varphi_X \circ d = X.
\]
Since for any $g\in \cin(\R^n)$
\[
g|_U = g\circ (x_1|_U, \ldots, x_n|_U) = g_{\cin(U)}(x_1|_U, \ldots, x_n|_U) 
\]
\[
  d(g|_U) = d \left(g_{\cin(U)}(x_1|_U, \ldots, x_n|_U) \right) =
  \sum (\partial_i g) _{\cin(U)} (x_1|_U, \ldots, x_n|_U)  dx_i|_U =
  \sum_i \partial_i g|_U dx_i|_U.
 \] 
 Therefore
 \[
\varphi_X ( d(g|_U)) = \sum \partial_i g|_U \,X(x_i|_U) = X(g|_U),
\]
where in the last equality we used the fact that $X$ is a
$\cin$-derivation.
However not all functions in $\cin(U)$ are restrictions of smooth
functions on $\R^n$.    For arbitrary functions in $\cin(U)$ we argue
as follows.
By the Localization Theorem of Mu\~noz D\'iaz and Ortega
\cite{MO} (see also \cite[p.\ 28]{NGSS})
given a function $f\in \cin(U)$ there exist $g, h\in \cin(\R^n)$ so that
\[
 h|_U f= g|_U
\]
and $h|_U$ is invertible in $\cin(U)$.   Then
\[
df = d\frac{g|_U}{h|_U}  = \frac{1}{h|_U} d (g|_U) -
\frac{g|_U}{(h|_U)^2} d (h|_U).
  \]
We now write $g$ and $h$ for $g|_U$ and $h|_U$ to reduce the clutter.
We have 
  \[
    \varphi_X (df) = \varphi_X (\frac{1}{h} d g-\frac{g}{h^2} d h) =
    \frac{1}{h} X(g)-\frac{g}{h^2}  X(h) \overset{\textrm{ Lemma~\ref{lem:B.2}}}{=} X(\frac{g}{h}) = X(f).
  \]
  And we are done.
\end{proof}  
\end{example}

We now prove existence of $\cin$-K\"ahler differentials.  
\begin{theorem}[Dubuc and Kock\cite{DK}] \label{thm:DK}
For any $\cin$-ring $\scC$ there exists a module of K\"ahler
differentials $\Omega_\scC^1$ together with the $\cin$-ring derivation
$d_\scC:\scC\to \Omega_\scC^1$, which has the appropriate universal properties.
\end{theorem}

Before we proceed with a proof of this theorem, we note that for any $\cin$-ring $\scC$
and any set $S$ there exists a free $\scC$-module $F[S]$ on the set
$S$.  The elements of $F[S]$ are finite formal linear combination
$a_1s_1+\cdots + a_n s_n$ for $n>0$, $a_1,\ldots, a_n\in \scC$, $s_1,\ldots,
s_n \in S$.

\begin{proof}[Proof of Theorem~\ref{thm:DK}] 
In commutative algebra there are two equivalent constructions of
K\"ahler differentials of a commutative ring.  Dubuc and Kock adapted
one of them to the $\cin$-ring setting (strictly speaking they adapted
it to the setting of Fermat theories, which include $\cin$-rings).  Joyce \cite{Joy} adapts the
other.  It will be convenient for us to follow Joyce's approach.

Consider the free $\scC$-module $F[\hd \scC]$
 where $\hd \scC$ is a
set of formal symbols $\hd b$, one symbol for each $b\in \scC$.  Thus
\[
F[\hd \scC] =\{ a_1\hd b_1+\cdots + a_n \hd b_n \mid n>0, a_1,\ldots,
a_n, b_1,\ldots, b_n \in \scC\}.
\]  
Let $I$ denote the submodule of $F[\hd \scC]$ generated by the set
\[
\scR:= \{ \hd (f_\scC(a_1,\ldots, a_n) )- \sum (\partial_i f)_\scC
(a_1,\ldots, a_n) \, \hd a_i \mid n>0, f\in \cin(\R^n), a_1,\ldots,
a_n \in \scC\},
\] 
that is, 
\[
  I = \langle \scR\rangle.
\]
Let
\[
\Omega^1_\scC : = F[\hd \scC] /I.
  \]
Then 
$\Omega^1_\scC$ is an $\scC$-module. Define $d_\scC : \scC\to
\Omega^1_\scC$ by
\[
d_\scC (a) = \hd a +I.
  \]
It is easy to check that $d_\scC$ is a $\cin$-derivation.    Note also
that the module $\Omega^1_\scC$ is generated by the set
$\{d_\scC(a)\}_{a\in \scC}$.

It remain to check the universal property of $d_\scC$. Let $\scN$ be
an $\scC$-module and $X:\scC\to \scN$ a $\cin$-derivation.  Since
$F[\hd \scC]$ is a free $\scC$-module on the set $\{\hd b\}_{b\in
  \scC}$, there is a unique map of $\scC$-modules $\hat{\varphi}:F[\hd
  \scC] \to \scN$ with
  \[
\hat{\varphi} (\hd b) = X(b).
\]
for all $b\in \scC$.
Since $X$ is a derivation,  $\hat{\varphi} (\alpha) = 0$ for all $\alpha $ in the
submodule $I$.  Consequently $\hat{\varphi}$ descends to a map of
$\scC$-modules
\[
\varphi_X:\Omega^1_\scC \to \scN.
  \]
  Then by construction of the map $\varphi_X$ and of the derivation $d: \scC\to \Omega^1_\scC$
  \[
\varphi_X(d_\scC b) = X(b) 
    \]
for all $b\in \scC$.  Since the set $\{d_\scC b\}_{b\in \scC}$
generates the module $\Omega^1_\scC$ of K\"ahler differentials, the map $\varphi_X$
is unique.
\end{proof}

\begin{remark} \label{rmrk:3.9aug}
If $M$ is a manifold then the $\cin(M)$-module $\bOmega^1_{dR}(M)$ of
ordinary one-forms on $M$ is the module of K\"ahler differentials
$\Omega_{\cin(M)}^1$.  This fact is stated in
\cite[Example~5.4]{Joy}.   In the more general case where $M$ is a
manifold with corners this fact is proved in \cite[Proposition~4.7.5]{FS}.
\end{remark}

\noindent
We next check that the universal derivation $d:\scC\to \Omega^1_\scC$ is natural
in the $\cin$-ring $\scC$.

\begin{lemma} \label{lem:5.12}
Let $f:\scC\to \scD$ be a map of $\cin$-rings.  Then there is a unique
map $\Lambda^1(f): \Omega_\scC^1\to \Omega^1_\scD$ of $\scC$-modules
making the diagram
\begin{equation} \label{eq:5:11}
\xy
(-15,10)*+{\scC}="1";
(15,10)*+{\Omega^1_\scC}="2";
(-15,-5)*+{\scD}="3";
(15,-5)*+{\Omega^1_\scD}="4";
{\ar@{->}^{d_\scC} "1";"2"};
{\ar@{->}_{f} "1";"3"};
{\ar@{-->}^{\Lambda^1(f)} "2";"4"};
{\ar@{->}_{d_\scD} "3";"4"};
\endxy
\end{equation}
commute.
\end{lemma}

\begin{proof} 
The composite $d_\scD \circ f :\scC \to \Omega^1_\scD$ is a
$\cin$-derivation.  By the universal property of $d_\scC$ there is a
unique map of $\scC$-modules $\Lambda^1(f): \Omega^1_\scC \to
\Omega^1_\scD$ making the diagram \eqref{eq:5:11} commute.
\end{proof}  
\begin{remark}
The notation $\Lambda^1(f)$ is chosen to be consistent  with the notation of Proposition~\ref{prop:6.5}.
\end{remark}
\noindent The next lemma shows that just like in the case of
commutative rings and K\"ahler differentials, the module of $\cin$-K\"ahler
differentials $\Omega^1_{\scC/I}$  of a quotient $\cin$-ring $\scC/I$  is a  quotient of
$\Omega^1_\scC$ by an appropriate submodule.  It is useful for computations.
\begin{lemma}\label{lem:5.5}  \label{lem:8.2}
Let $f:\scC \to \scD$ be a surjective map of
$\cin$-rings, $I= \ker f$ and $\Lambda^1(f):\Omega^1_\scC \to
\Omega^1_\scD$ the induced map on K\"ahler differentials.   Then
$\Lambda^1(f)$ is surjective and $\ker \Lambda^1(f) = J$, where
\[
J = \scC d_\scC I + I d_\scC \scC = \scC d_\scC I.
\]  
\end{lemma}

\begin{proof}
Since $\Omega^1_\scC$, $\Omega^1_\scD$ are generated by the sets
$\{d_\scC c\}_{c\in \scC}$ and $\{d_\scD x\}_{x\in \scD}$,
respectively,  since $f$ is surjective and since $\Lambda^1(f)(d_\scC
c) = d_\scD(f(c))$ for all $c\in \scC$, $\Lambda^1(f)$ is surjective.

We next drop the subscript $_\scC$ on $d_\scC$ and argue that $\scC dI + I d\scC = \scC dI$. For any $c\in \scC$
and any $i\in I$, $d(ic)\in dI$.  Hence $dI \ni d(ic) = c di + idc$.
Consequently $i\,dc = d(ci) - cdI \in \scC dI$.  It follows that
$I d\scC \subseteq \scC dI$.  We conclude that $J = \scC dI$.

Since \eqref{eq:5:11} commutes and $\Lambda^1(f)$ is a map of
$\scC$-modules, $J= \scC dI \subseteq \ker \Lambda^1(f)$.  We need to
show that $J = \ker \Lambda^1(f)$.

Note that for any $i\in I$ and any $a\in J$, the product $ia$ is in
$J$.  Therefore the quotient $\scC$-module $\Omega^1_\scC/J$ is
naturally a $\scD$-module with the action of $\scD$ given by
\[
f(c) \cdot (\alpha +J) = (c \alpha) +J
\]  
for all $c\in \scC$ and all $\alpha + J \in \Omega^1_\scC/J$.
The map 
\[
\bar{f}:\Omega^1_\scC/J\to \Omega^1_\scD,\qquad \bar{f} (\alpha +J):=
\Lambda^1 (\alpha) 
\]  
is $\scD$-linear.  We will argue that $\bar{f}$ is an isomorphism of
modules, which is enough to prove that $\ker \Lambda^1(f) = J$ and
thereby finish the proof of the lemma.

Let $p: \Omega^1_\scC \to \Omega^1_\scC /J$ denote the canonical projection.  The map $p$ 
is a map of $\scC$-modules.  Then $p\circ d_\scC: \scC \to
\Omega^1_\scC /J$ is a $\cin$-derivation of the $\cin$-ring $\scC$.  Since for any $i\in I$,
$d_\scC(i) \in J$, the composite $p\circ d_\scC$ induces a map
\[
\bar{d}:\scD\to \Omega^1_\scC/J
\]
making the diagram
\[
\xy
(-15,10)*+{\scC}="1";
(15,10)*+{\Omega^1_\scC}="2";
(-15,-5)*+{\scD}="3";
(15,-5)*+{\Omega^1_\scC/J}="4";
{\ar@{->}^{d_\scC} "1";"2"};
{\ar@{->}_{f} "1";"3"};
{\ar@{->}^{p} "2";"4"};
{\ar@{->}_{\bar{d}} "3";"4"};
\endxy
\]
commute, i.e.,
\[
\bar{d} (f(c))  = p(d_\scC(c)) = d_\scC(c)+J
\]
for all $c\in \scC$.  A computation shows that $\bar{d}$ is a
$\cin$-derivation of the $\cin$-ring $\scD$.  By the universal
property of $d_\scD:\scD \to \Omega^1_\scD$ there exists a unique
$\scD$-linear map $\psi: \Omega^1_\scD\to \Omega^1_\scC/J$ such that
\begin{equation} \label{eq:3.3aug}
\bar{d}(x) = \psi (d_\scD(x))
\end{equation} 
for all $x\in \scD$.  Since the set $\{d_\scC(c)\}_{c\in \scC}$
generates the module $\Omega^1_\scC$ and since $p$ is onto, the set
\[
\{p(d_\scC(c))\}_{c\in \scC} = \{ \bar{d}(c)\}_{c\in \scC} =
\{\bar{d}(x)\}_{x\in \scD}
\]  
generates the module $\Omega^1_\scC/J$ (both over $\scC$ and over $\scD$).
Equation \eqref{eq:3.3aug} now implies that $\psi$ is onto.

Since $J\subset \ker \Lambda^1(f)$, the map $\Lambda^1(f)$ induces a
map $\bar{f}:\Omega^1_\scC/J \to \Omega^1_\scD$ so that
\[
\bar{f}\circ p = \Lambda^1(f).
\]  
For any $c\in \scC$
\begin{equation} \label{eq:3.4aug}
\bar{f} (\bar{d}(f(c)))= \bar{f} (p(d_\scC(c))) = \Lambda^1(f)
(d_\scC(c)  = d_\scC(f(c)).
\end{equation}
Since $f$ is onto \eqref{eq:3.4aug} implies that
\begin{equation} \label{eq:3.5aug}
\bar{f}\circ \bar{d} = d_\scD.
\end{equation}
Since $\bar{d} = \psi \circ d_\scD$ equation \eqref{eq:3.5aug} implies that 
\[
(\bar{f}\circ \psi)(d_\scD(x)) = d_\scD(x)
\]
for all $x\in \scD$, 
 Since the set
 $\{d_\scD(x)\}_{x\in \scD}$ generates $\Omega^1_\scD$ we conclude that
\[
\bar{f}\circ \psi = \id_{\Omega^1_\scD}.
\]
Hence $\psi$ is injective.   Since we have already shown that $\psi$
is onto, this implies that $\psi$ is an isomorphism.  But then
$\bar{f}$ has to be an isomorphism as well, and we are done.
\end{proof}  
\begin{example} \label{ex:Cantor}
Let $C\subset \R$ denote the standard Cantor set (middle third).
Since $C $ is closed in $\R$, the set $\cin(C):= \cin(\R)|_C$ of
restrictions of smooth functions to $C$ is the induced subspace
differential structure on $C$ (see
Definition~\ref{def:induced_diff_str} and Lemma~\ref{lem:2.27}).

\begin{claim*}
The module $\Omega^1_{\cin(C)}$ of $\cin$-K\"ahler differentials of
the $\cin$-ring $\cin(C)$ is a free $\cin(C)$-module
generated by $d(x|_C)$.  Here $x:\R\to \R$ is the identity map.
\end{claim*}  

\begin{proof}
By Example~\ref{ex:Krn}, the universal derivation $d_{\cin(R)} :\cin(\R) \to
\Omega^1_{\cin(\R)}$ is the ordinary exterior derivative $d:\cin(\R)
\to \bOmega^1_{dR} (\R)$.   The restriction map $|_C: \cin(\R) \to
\cin(\R)|_C \equiv \cin(C)$ is surjective.   It induces a surjective
map $\bOmega^1_{dR} (\R) \to \Omega^1_{\cin(C)}$, which we again
denote by $|_C$ (this abuse of notation is  convenient).  By
Lemma~\ref{lem:5.5}
\[
J:= \ker \left(|_C :\bOmega^1_{dR} (\R) \to \Omega^1_{\cin(C)}\right) = \cin(\R) dI +
I d \cin(\R),
\]  
where $I$ is the ideal of functions on $\R$ that vanish on the Cantor
set:
\[
I = \{f\in \cin(\R) \mid f|_C =0\}.
\]
Since $dx$ generated $\bOmega^1_{dR}(\R)$ its restriction $dx|_C$ generates
$\Omega^1_{\cin(C)} = \bOmega^1_{dR}(\R)/J$.  We now argue that
$\Omega^1_{\cin(C)}$ is freely generated by $dx|_C$.  That is,
\[
f|_C \, dx|_C = 0 \quad \Rightarrow f|_C = 0.
\]  
Since the Cantor set has no isolated points, for any functions $f$
vanishing on $C$
the derivative $f'$ also vanishes on $C$. Hence for any $f\in I$
\[
d f = f' dx \in I d\cin(\R).
\]
Consequently $\cin(\R) dI \subseteq  I d\cin (\R)$ and therefore  $J = I
d\cin(\R)$.  Since $
d \cin (\R) = \cin(\R) dx$
\[
J  = I d \cin(\R)  = I \cin(\R) dx = I dx.
\]
Now suppose $f|_C \, d(x|_C) = 0$ for some $f\in \cin(\R)$.  Then
$0 = (fdx)|_C $ and therefore $fdx \in J = I dx$.  Since
$\bOmega^1_{dR} (\R)$ is a free $\cin(\R)$-module generated by $dx$,
$f\in I$.  Therefore $f|_C =0$ in $\cin(C)$.
\end{proof}
\end{example}
  
Our next example, Example~\ref{ex:xy=0}, will produce  a $\cin$-ring $\scC$ so that the map
\[
\psi: \Omega^1_\scC \to \Hom (\cDer(\scC), \scC)
\]
of Remark~\ref{rmrk:3.3} is not injective.    In order to prove that this
is indeed the case, we need two lemmas.

\begin{lemma} \label{lem:3.14aug}
Let $f:\scC\to \scD$ be a map of $\cin$-rings,  $v\in \cin\Der(\scC)$,
$w\in \cin\Der(\scD)$ two derivations with $f\circ v = w\circ f$.
Let $\imath_v: \Omega^1_\scC \to \scC$, $\imath_w: \Omega^1_\scD\to
\scD$ be the corresponding contractions (Remark~\ref{rmrk:3.3}).
Then the diagram
\begin{equation}\label{eq:3.3.1june}
\xy
(-15,10)*+{\Omega^1_\scC}="1";
(15,10)*+{\scC}="2";
(-15,-5)*+{\Omega^1_\scB }="3";
(15,-5)*+{\scD}="4";
{\ar@{->}^{\imath_v} "1";"2"};
{\ar@{->}_{\Lambda^1 f} "1";"3"};
{\ar@{->}^{f} "2";"4"};
{\ar@{->}_{\imath_w} "3";"4"};
\endxy
\end{equation}
commutes. 
\end{lemma}  

\begin{proof}
Since $f$ is a $\cin$-ring homomorphism, $w\circ f = f\circ
v:\scC\to \scD$ is a $\cin$-ring derivation of $\scC$ with values in
$\scD$ (thought of as a $\scC$-module by way of $f$).  By the universal
property of $d_\scC:\scC \to \Omega^1_\scC$ there exists a unique
$\scC$-linear map $\psi:\Omega^1_\scC\to \scD$  so that
\begin{equation} \label{eq:3.3.2june}
f\circ v = \psi \circ d_\scC = w \circ f.
\end{equation}

By definition of the contraction $\imath_v$ we have
\[
f\circ v =
f\circ \imath_v\circ d_\scC.
\]
By definitions of $\imath_w$ and of the map $\Lambda^1f$
\[
w\circ f  =(\imath_w \circ d_\scD) \circ f = \imath_w \circ
\Lambda^1 f \circ d_\scC.
\]  
Therefore by the uniqueness of the map $\psi$ satisfying
\eqref{eq:3.3.2june},
\[
f \circ \imath_v = \imath_w \circ f,
 \] 
i.e., \eqref{eq:3.3.1june} commutes.
\end{proof}

\begin{lemma} \label{lem:3.15aug}
Let $M\subset \R^n$ be a closed subset and $\scF= \cin(\R^n)|_M$ the
induced subspace differential structure.   For any derivation $v\in
\cDer(\scF)$ there exists a (non-unique) derivation $V\in \cin(\R^n)$,
i.e. a vector field on $\R^n$, so that
\begin{equation}  \label{eq:3.8}
V(f)|_M = v(f|_M)
\end{equation}
for all $f\in \cin(\R^n)$.  Here, as in Example~\ref{ex:Cantor},
$|_M:\cin(\R^n) \to \scF = \cin(\R^n)|_M$ is the restriction map.
\end{lemma}

\begin{proof}
Since the set $\{x_1, \ldots, x_n\}$ of coordinate functions on $\R^n$
generate $\cin(\R^n)$ as a $\cin$-ring and since the restriction map
$|_M:\cin(\R^n)\to \scF$ is surjective, the set $\{x_1|_M, \ldots,
x_n|_M\}$ generates $\scF$.   More explicitly,  for any $h\in
\cin(\R^n)$
\[
  h_\scF(x_1|_M, \ldots,x_n|_M) = h\circ (x_1|_M, \ldots,x_n|_M)
  =\left( h \circ (x_1,\ldots, x_n)\right)|_M = h|_M.
\]
Since $v$ is a $\cin$-derivation, for any $f\in \cin(\R^n)$,
\[
v (f|_M) = v(f_\scF( x_1|_M, \ldots,x_n|_M)) = \sum_i (\partial_i
f)_\scF (x_1|_M, \ldots,x_n|_M)\cdot  v(x_i|_M).
\]
Since each $v(x_i|_M) \in \scF = \cin(\R^n)|_M$ there exist smooth
functions $a_1,\ldots, a_n\in \R^n$ with $a_i|_M =  v(x_i|_M)$ for all
$i$ (these functions are only unique up to the ideal $I_M$ of
functions vanishing on $M$).  Now define a vector field $V$ on
$\R^n$ to be
\[
V := \sum_i a_i \frac{\partial}{\partial x_i}.
\]  
For any function $f\in \cin(\R^n)$
\[
  \begin{split}
V(f)|_M &= \left.\left( \sum a_i \partial_if) \right) \right|_M =  \sum
(\partial_if)|_M \cdot a_i |_M \\
&= (\partial_if) _\scF( x_1|_M, \ldots,x_n|_M) \cdot v (x_i|_M)\\
&= v (f_\scF(x_1|_M, \ldots,x_n|_M) )\qquad \textrm{ (since $v$ is a
  $\cin$-derivation)}\\
&= v(f|_M).
  \end{split}
\]  
\end{proof}  

\begin{remark}
Given $M\subset \R^n$ as in Lemma~\ref{lem:3.15aug}, let $I_M $ denote
the ideal of functions on $\R^n$ that vanish on $M$. Then
$\cin(\R^n)|_M = \cin(\R^n)/I_M$ and
\eqref{eq:3.8} can be restated as follows:
\[
  V(f) +I_M = v(f+I_M)
\]
for all $f\in \cin(\R^n)$.
\end{remark}  

\begin{example} \label{ex:5.15} \label{ex:xy=0} Consider the closed
  subset $M$ of $ \R^2$ consisting of the union of two coordinate
  axes:
\[
M = \{(x,y)\in \R^2 \mid xy =0\}.
\]
Then $\cin(M) := \cin(\R^2)|_M$ is the induced differential space
structure on $M$.  As in Example~\ref{ex:Cantor} and in
Lemma~\ref{lem:3.15aug}, let $|_M: \cin(\R^2)\to \cin(M)$ denote the
restriction map.  The kernel of $|_M$ is the ideal
\[
I = \{f\in \cin(\R^2) \mid f|_M = 0\}.
\]  
Using Hadamard's lemma \cite[p.\ 17]{jet} twice it is not hard to show
that $I$ is generated by the function $xy$.  Here are the details.
Suppose  $f(x,y) \in I$.  Then $f(0,y) =
0$ for all $y$ hence by Hadamard
\[
f(x,y) = f(x,y) - f(0,y) = x g(x,y)
\]    
for some $g\in \cin(\R^2)$.  Since $f(x,0) = 0$ for all $x$, $g(x,0) =
0$ for all $x\not = 0$ and, by continuity, $g(x,0) = 0$ for all $x$.
Therefore
\[
g(x,y) = g(x,y) - g(x,0) = y h(x,y)
\]  
for some $h\in \cin(\R^2)$.  Hence $f(x,y) = xy h(x,y)$ for some $h\in
\cin (\R^2)$.
Thus $\cin(M) =
\cin(\R^2) /\langle xy \rangle $.  By Lemma~\ref{lem:5.5}
\[
\Omega^1_{\cin(M)} = \Omega^1_{\cin(\R^2)}/ J 
\]
where
\[
  J = \cin(\R^2) \, d  I +I\, d (\cin(R^2)) = \{ f d(xy) + xy g \, dx
  + xy h\, dy\mid f,g,h\in \cin(\R^2)\}.
\]
Since $\Omega^1_{\cin(\R^2)}=\bOmega^1_{dR}(\R^2)$ is freely generated by $dx, dy$, the
module $\Omega^1_{\cin(M)} = \Omega^1_{\cin(\R^2)}/J$ is generated by
$dx+J$ and $dy+J$ which we  write as $dx|_M$ and $dy|_M$,
respectively.   Note that since $ 0  = d 0 = d(xy|_M)= d(xy)|_M = x|_M dy|_M +y|_M dx|_M$
there is a relation between the generators.

On the other hand $xdy \not \in J$ so $xdy|_M \not =0$ in
$\Omega^1_{\cin(M)}$.  We now argue that for any derivation $v\in
\cDer(\cin(M))$
\[
\imath_v (xdy|_M) = 0
\]  
(we use the notation of Lemma~\ref{lem:3.14aug}). Fix a derivation
$v\in \cDer(\cin(M))$.  By Lemma~\ref{lem:3.15aug} there exists a
vector field $V=a_1\frac{\partial}{\partial x} + a_2 \frac{\partial}{\partial y}$ on $\R^2$ so that $V(f)|_M = v(f|_M)$ for all $f\in \cin(\R^n)$.
Then for $f(x,y) = xy$ we get
\[
0= v(0) =  v(xy|_M) = V(xy)|_M = (y a_1 + x a_2)|_M.
\] 
Consequently  $ya_1 + xa_2 = xy a_3$ for some $a_3\in \cin(\R^2)$.
Then $0 = y a_1(0,y)$ for all $y$.  Therefore $a_1(0, y) = 0$ for all
$y\not =0$ and thus $a_1(0,y) = 0$ for all $y$ by continuity.
Hadamard's lemma now implies that $a_1(x,y) = x b_1(x,y)$ for some
$b_1\in \cin(\R^2)$.  Similarly $a_2(x,y) = y b_2(x,y)$ for some
$b_2\in \cin(\R^2)$.  In other words $V$ is tangent to the coordinate
axis and vanishes at the origin.  By Lemma~\ref{lem:3.14aug}
\[
\imath_V (xdy))|_M = \imath_v ((xdy)|_M) 
\]  
Since $\imath_V (xdy) = xyb_2$, the restriction $\imath_V (xdy))|_M $
is zero.  This proves that $ \imath_v ((xdy)|_M) = 0$ for all
derivations $v$.
\end{example}

\section{$\cin$-algebraic  de Rham complex} \label{sec:drh}

In this section by analogy with the algebraic de Rham complex
\cite{Groth, Groth67} we construct the $\cin$-algebraic de Rham
complex $(\Lambda ^\bullet \Omega^1_\scC, \wedge, d)$ of a $\cin$-ring $\scC$.

\begin{notation} Let $\Omega^1_\scC$ be the module of
  $\cin$-K\"ahler differentials of a $\cin$-ring $\scC$.  For $k\in \Z$ we
 denote the  $k$-th exterior power of the $\scC$-module
 $\Omega^1_\scC$ by $\Lambda^k (\Omega^1_\scC)$.  By convention 
\[
\Lambda^0 (\Omega^1_\scC):= \scC \qquad \textrm{ and } \Lambda^k
(\Omega^1_\scC) = 0\textrm{ for } k<0.
 \]
 We set
 \[
\Lambda^\bullet\Omega^1_\scC := \{ \Lambda^k\Omega^1_\scC\}_{k\in \Z}
\]
which is a commutative graded algebra over the $\cin$-ring $\scC$.
\end{notation}

We now turn  the commutative graded algebra $\Lambda ^\bullet \Omega^1_\scC$ into a commutative 
  {\em differential}  graded algebra (a CDGA).
\begin{theorem}\label{thm:6.2}
Let $\scC$ be a $\cin$-ring.  The universal differential $d_\scC:\scC
\to \Omega^1_\scC$ extends uniquely to a degree 1 derivation $d\in \Der(\Lambda ^\bullet \Omega^1_\scC)$
with the property that $d\circ d=0$.  

In particular, for any $k>0$ and any $a, b_1,\ldots, b_k\in \scC$
\begin{equation} \label{eq:6.d}
d(a \, d_\scC b_1\wedge  \ldots \wedge d_\scC b_k) = d_\scC a \wedge d_\scC b_1\wedge  \ldots \wedge d_\scC b_k.
\end{equation}

\end{theorem}
Our proof  mimics  \cite[\href{https://stacks.math.columbia.edu/tag/0FKL}{Tag
  0FKL}]{stacks-project}.  As a first step we prove
\begin{lemma} \label{lem:4.3aug}
Let $\scC$ be a $\cin$ ring.  The universal differential $d_\scC:\scC
\to \Omega^1_\scC$ gives rise to a unique $\R$-linear map $d:\Omega^1_\scC
\to \Omega^1_\scC \wedge \Omega^1_\scC = \Lambda^2\Omega^1_\scC$
with
\begin{itemize}
\item $d(\sum_{i=1}^k a_i d_\scC b_i ) = \sum_i  d_\scC a_i\wedge  d_\scC b_i $
  for all $k>0$, all $a_1,\ldots, a_k, b_1,\ldots, b_k \in \scC$ and
\item   $d(d_\scC a) = 0$ for all $a\in\scC$.
\end{itemize}  
\end{lemma}  
\begin{proof}
It is easy to see that if such map $d$ exists, then it has to be unique.  To
prove existence of $d$ consider the $\R$-linear map
\[
  \varphi: F(\hd \scC) \to \Omega^1_\scC \wedge \Omega^1_\scC, \qquad
  \varphi (\sum_{i=1}^k a_i \hd b_i ) = \sum_i  d_\scC a_i\wedge  d_\scC b_i ,
  \]
  where, as in the proof of Theorem~\ref{thm:DK},  $F(\hd \scC)$ is a
  free $\scC$-module on the set $\{\hd b\}_{b\in \scC}$.  Since
  $d_\scC$ is a derivation, $d_\scC(1) = 0$ and consequently 
  \[
    \varphi(\hd b) = \varphi(1\, \hd b) = d_\scC (1) \wedge d_\scC b =
    0.
  \]
  We check that $\varphi$ vanishes on the submodule
  $I =\langle \scR\rangle$ of $F(\hd \scC)$ (see the proof of
  Theorem~\ref{thm:DK} for the notation) hence descends to a map
  $d: \Omega^1_\scC = F(\hd \scC)/I\to \Omega^1_\scC \wedge
  \Omega^1_\scC$.  This is a computation that uses the fact that mixed
  partials commute.  Given $n>0$, $f\in \cin(\R^n)$ and
  $a_1,\ldots, a_n\in \scC$
\[
\begin{split}
\varphi &\left( \hd (f_\scC(a_1,\ldots, a_n) - \sum (\partial_i f)_\scC
(a_1,\ldots, a_n) \, \hd a_i \right)\\
=& 0 - \sum_i d_\scC \left((\partial_i f)_\scC(a_1,\ldots, a_n)
\right)\wedge  d_\scC  a_i \qquad \textrm{ ( since } \varphi(\hd  b) =0
\textrm{ for all } b\in \scC)\\
=&-\sum_{j,i} (\partial_{ji}^2 f)_\scC(a_1,\ldots, a_n) d_\scC
a_j\wedge  d_\scC  a_i\\
=& \sum_{i<j}\left( (\partial_{ij}^2 f)_\scC(a_1,\ldots, a_n)
   -(\partial_{ji}^2 f)_\scC(a_1,\ldots, a_n)     \right)\,
d_\scC a_i
\wedge d_\scC a_j \\
=& \sum_{i<j} (\partial_{ij}^2 f -\partial_{ji}^2 f)_\scC(a_1,\ldots, a_n)  \,
d_\scC a_i\wedge d_\scC a_j  =0.
\end{split}
\]

\end{proof}

\begin{proof}[Proof of Theorem~\ref{thm:6.2}] It remains to construct
\[
d: \Lambda^p\Omega^1_\scC \to \Lambda^{p+1}\Omega^1_\scC
\]      
for $p>1$, show that $d\circ d =0$ and prove that $d$ is a derivation
of degree 1.   The construction of $d$ is essentially identical to the one in
\cite{stacks-project}.   Consider the map
\[
\beta: \overbrace{\Omega_\scC^1\times \cdots \times \Omega_\scC^1}^{p}
\to \Lambda^{p+1} \Omega^1_\scC, \qquad \beta(\alpha_1,\ldots, \alpha_p) = \sum
_i (-1)^{i+1} \alpha_1\wedge \ldots\wedge d_\scC \alpha_i \wedge
\ldots\wedge \alpha_p.
 \] 
 Since $\beta$ is $\R$-multilinear it descends to an $\R$-linear map
 \[
\gamma: \Omega_\scC^1\otimes_\R\cdots \otimes_\R \Omega_\scC^1 \to
\Omega^{p+1}_\scC, \qquad \gamma(\alpha_1\otimes\cdots \otimes \alpha_p) = \sum
_i (-1)^{i+1} \alpha_1\wedge \ldots\wedge d_\scC \alpha_i \wedge \ldots\wedge \alpha_p.
 \]  
 We will show that $\gamma$ factors through the natural surjection
\[
\Omega_\scC^1\otimes_\R\cdots \otimes_\R \Omega_\scC^1 \to
\Lambda^{p}\Omega^1_\scC %
\]
to give the desired map
$d:\Lambda^p \Omega^1_\scC \to \Lambda^{p+1}\Omega^1_\scC$ with
\begin{equation} \label{eq:*p27}
d(\alpha_1\wedge \ldots \wedge \alpha_p) = \sum
_i (-1)^{i+1} \alpha_1\wedge \ldots\wedge d_\scC \alpha_i \wedge
\ldots\wedge \alpha_p
\end{equation}
for all $\alpha_1, \ldots, \alpha_p\in \Omega^1_\scC$.
It is easy to see that $\gamma$ is alternating.  It remains to check
that for any two indices $i\not =j$, any $a\in \scC$ and any
$\alpha_1,\ldots, \alpha_p\in \Omega^p_\scC$
\begin{equation}\label{eq:7.1}
\gamma( \alpha_1\otimes\cdots \otimes a\alpha_i\otimes \cdots \otimes \alpha_p -
\alpha_1\otimes\cdots\otimes a\alpha_j\otimes \cdots
\otimes \alpha_p) = 0.
\end{equation}
Now
\[
  \begin{split}
\gamma(\alpha_1\otimes\cdots \otimes a\alpha_i\otimes \cdots \otimes
\alpha_p )  &= a \gamma (\alpha_1\otimes\cdots \otimes
\alpha_p ) + (-1) ^{i+1} \alpha_1\wedge \ldots \wedge
\alpha_{i-1}\wedge da \wedge \alpha_i\wedge \ldots\wedge \alpha_p\\
&= a\gamma (\alpha_1\otimes\cdots \otimes
\alpha_p ) + da\wedge \alpha_1\wedge \ldots \wedge \alpha_p.
\end{split}
\]
Hence \eqref{eq:7.1} holds.  This proves the existence of the map
$d:\Lambda^p\Omega^1_\scC \to \Lambda^{p+1}\Omega^1_\scC$ satisfying \eqref{eq:*p27}.
By Lemma~\ref{lem:4.3aug}, $d(d_\scC a)= 0$ for  all $a\in \scC$.
Therefore \eqref{eq:*p27} implies that
\[
  \begin{split}
    d(ad_\scC b_1\wedge \ldots \wedge d_\scC b_k ) &=  d(ad_\scC b_1)
    \wedge d_\scC b_2 \ldots \wedge d_\scC b_k\\
    &= d_\scC a \wedge d_\scC b_1\wedge \ldots \wedge d_\scC b_k
\end{split}
\]
for all $k\geq 1$ and all $a,b_1,\ldots, b_k\in \scC$.  Hence
\eqref{eq:6.d} holds.
Since $d_\scC (1) = 0$, it follows that 
\[
d(  d_\scC b_1\wedge  \ldots \wedge d_\scC b_k) =d(  1\, d_\scC b_1\wedge  \ldots \wedge d_\scC b_k) = 0.
\]
Consequently
\[
(d \circ d) \,(a\,\,   d_\scC b_1\wedge  \ldots \wedge d_\scC b_k) = d(
d_\scC a \wedge   d_\scC b_1\wedge  \ldots \wedge d_\scC b_k) = 0.
\]  
We conclude that $d\circ d = 0$.

It remains to show that
\begin{equation} \label{eq:6d2}
d (\alpha \wedge \beta) = (d\alpha) \wedge \beta + (-1)^{|\alpha|}\alpha
\wedge d \beta
\end{equation}  
for any $\alpha, \beta \in \Lambda^\bullet
\Omega^1_\scC$ (here as before $|\alpha|$ is the degree of $\alpha$).
It is no loss of generality to assume that $\alpha = a_0 \, d_\scC a_1\wedge
\ldots \wedge d_\scC a_k$ and  $\beta =b_0\, d_\scC b_1\wedge \ldots \wedge
d_\scC b_\ell$ for some $k,\ell \geq 0$ and some $a_0, \ldots b_\ell \in \scC$.   Then 
\[
\alpha \wedge\beta = a_0 b_0\, d_\scC a_1\wedge \ldots  d_\scC a_k
\wedge d_\scC b_1 \wedge \ldots \wedge d_\scC b_\ell
\]
and \eqref{eq:6d2} follows easily from that fact that  \eqref{eq:6.d} holds.
\end{proof}

It will be useful to know that 
the $\cin$-algebraic de Rham complex $\Lambda^\bullet\Omega^1_\scC $  is
functorial in the $\cin$-ring $\scC$:
\begin{proposition} \label{prop:6.5}
A map $f:\scC\to \scB$ of $\cin$-rings induces a unique map
$\Lambda^\bullet(f): \Lambda^\bullet\Omega^1_\scC \to \Lambda^\bullet\Omega^1_\scB$ of
commutative differential graded algebras.  Explicitly
\[
  \Lambda^\bullet(f) (a_0 d_\scC a_1\wedge \cdots \wedge d_\scC a_k)
  = f(a_0) d_\scB f(a_1)\wedge \cdots \wedge d_\scB f(a_k)
\]  
for all $k>0$ and all $a_0,\ldots, a_k\in \scC$.
\end{proposition}
\begin{proof}
By Lemma~\ref{lem:5.12} the map $f:\scC\to \scB$ induces a map
$\Lambda^1(f): \Omega^1_\scC\to \Omega^1_\scB$. By taking exterior
powers of $\Lambda^1(f)$
we get $\Lambda ^k(f): \Lambda^k\Omega^1_\scC \to \Lambda^k\Omega^1_\scB$ for each $k\geq
1$.  Together these maps give us a map of commutative graded algebras
\[
\Lambda^\bullet(f) := \Lambda^\bullet \Omega^1_\scC \to \Lambda ^\bullet\Omega^1_\scB
\]
with the property that
\[
\Lambda^\bullet(f) (a_0 d_\scC a_1\wedge \cdots \wedge d_\scC a_k)
  = f(a_0) d_\scB f(a_1)\wedge \cdots \wedge d_\scB f(a_k)
\]
for all $k>0$ and all $a_0,\ldots, a_k\in \scC$ (since $d_\scB (f(a))
= \Lambda^1(f) \left(d_\scC a\right)$ for all $a\in \scC$).
Now
\[
\begin{split}  
d \left( \Lambda^\bullet(f)\, (a_0 d_\scC a_1\wedge \cdots \wedge d_\scC
  a_k)\right)
&=d\left( f(a_0)\wedge  d_\scB f(a_1)\wedge \cdots \wedge d_\scB f(a_k)\right)\\
&=d_\scB f(a_0)\wedge  d_\scB f(a_1)\wedge \cdots \wedge d_\scB f(a_k)\\
&= \Lambda^\bullet(f) \left( d_\scC a_0 \wedge  d_\scC a_1\wedge \cdots \wedge d_\scC
 a_k\right)\\
&= \Lambda^\bullet(f) \left( d \left(  a_0 \, d_\scC a_1\wedge \cdots \wedge d_\scC
 a_k   \right)\right).
\end{split}
\]
Hence $\Lambda^\bullet(f) $ also commutes with $d$
and we are done.
\end{proof}

\section{Differential forms on $\cin$-ringed spaces}
\label{sec:forms}

In this section we construct sheaves of commutative differential
graded algebras over local $\cin$-ringed spaces.
We think of sections of these sheaves as
differential forms, since for manifolds with corners they are sheaves
of ordinary
differential forms. We show that our differential forms pull back,
that pull back commutes with the exterior derivatives and exterior
multiplication and that our differential forms can be integrated over
simplices and, more generally, over smooth singular chains.
Consequently a version of Stokes' theorem holds for these differential
forms. We start by constructing the appropriate presheaves.

\begin{lemma}  \label{constr:9.3}Let $(M, \scA)$ be a $\cin$-ringed
  space. There exists a presheaf $(\Lambda^\bullet \Omega^1_\scA ,
  \wedge, d)$ of commutative differential graded algebras over $\scA$ so that for each
open subset $U\subset M$ the CDGA $(\Lambda^\bullet \Omega^1_\scA (U)
, \wedge_U, dU)$ is the $\cin$-algebraic de Rham complex of the
$\cin$-ring $\scA(U)$.
\end{lemma}
\begin{proof}

For every open set $U\subset M$ we have an $\scA(U)$-module
$\Omega^1_{\scA(U)}$ of K\"ahler differentials, which comes with the
universal derivation $d:\scA(U) \to \Omega^1_{\scA(U)}$.   By
Lemma~\ref{lem:5.12} the assignment $U\mapsto \Omega^1_{\scA(U)}$ is
a presheaf of $\scA$-modules.  We denote this presheaf by
$\Omega^1_\scA$. Consequently for all $k\in \Z$, $U\mapsto \Lambda^k
\Omega^1_{\scA(U)}$ is also a presheaf of $\scA$-modules.  We denote it
by $\Lambda^k \Omega^1_\scA$.   Proposition~\ref{prop:6.5} implies for
all $k$ and $\ell$
that we have maps
\[
\wedge: \Lambda^k\Omega^1_\scA \times \Lambda^\ell \Omega^1_\scA\to
  \Lambda^{k+\ell}\Omega^1_\scA\qquad \textrm{and}\qquad d:
  \Lambda^k\Omega^1_\scA \to   \Lambda^{k+1} \Omega^1_\scA 
\]  
of presheaves of $\scA$-modules and that these maps turn the sequence
$\{\Lambda^k\Omega^1_\scA\}_{k\in \Z}$ of presheaves into a presheaf
  of commutative graded $\scA$-algebras.  
  We thus get a presheaf of
  CDGAs $(\{\Lambda^k\Omega^1_\scA\} _{k\in \Z}, \wedge, d)$ over the
  sheaf $\scA$.
\end{proof}

\begin{notation}
Given a $\cin$-ringed space $(M, \scA)$ we denote the presheaf of
CDGAs produced  by Lemma~\ref{constr:9.3} either by $(\Lambda^\bullet
\Omega^1_\scA, d, \wedge)$ or simply by $\Lambda^\bullet
\Omega^1_\scA$ with  $d$ and $\wedge$ understood.
\end{notation}  

As was mentioned in Remark~\ref{rmrk:3.9aug}, if $M$ is a manifold with
corners then the presheaf $\Omega^1_{\cin_M}$ of K\"ahler
differentials on $M$ agrees with the sheaf of sections
$\bOmega^1_{dR, M}$ of the cotangent bundle $T^*M\to M$; this is a
result of Francis-Staite \cite[Proposition~4.7.5]{FS}. Hence in
particular the presheaf $\Omega^1_{\cin_M}$ is a sheaf.  We now argue
that the whole presheaf
$(\Lambda^\bullet \Omega^1_{\cin_M}, d, \wedge)$ of $\cin$-algebraic
de Rham forms is a sheaf:

\begin{lemma} \label{lem:CDGAcorners}
Let $M$ be a manifold with corners.  Then the presheaf $(\Lambda^\bullet \Omega^1_{\cin_M},
  d, \wedge)$ of $\cin$-algebraic de Rham forms is 
  the 
  sheaf  $(\bOmega_{dR,M}^\bullet, d, \wedge)$ of ordinary de
  Rham differential forms. 
\end{lemma}

\begin{proof}
Given a vector bundle $F$ over a manifold with corners $M$ denote temporarily its sheaf of
sections by $\scS_F$.

For a vector bundle $E\to M$ over a (second countable
Hausdorff) manifold with corners $M$ taking $k$th exterior power of the bundle commutes
with taking global sections: there is an isomorphism of $\cin_M(M)$-modules
\begin{equation}
\varphi_M: \Lambda^k \left(\Gamma(E)\right)\to \Gamma(\Lambda^k E)
\end{equation}
with the property that for any $k$ global sections $s_1,\ldots, s_k$
the section $\varphi(s_1\wedge\cdots \wedge s_k)$ is the section
defined by
\[
q\mapsto s_1(q)\wedge \cdots \wedge s_k(q)
\]
(the wedge on the right is in the exterior algebra  $\Lambda^\bullet
E_q$ of the fiber $E_q$ of $E$).   See, for example, \cite[Section 7.5]{Conlon}.
If $V\subset U\subset M$ are two open sets in $M$ then
the diagram
\[
\xy
(-15,10)*+{\Lambda^k \Gamma(E|_U)}="1";
(15,10)*+{\Gamma(\Lambda^k(E|_U))}="2";
(-15,-5)*+{\Lambda^k \Gamma(E|_V)}="3";
(15,-5)*+{\Gamma(\Lambda^k(E|_V))}="4";
{\ar@{->} ^{\varphi_U} "1";"2"};
{\ar@{->}_{|_V} "1";"3"};
{\ar@{->} _{|_V}"2";"4"};
{\ar@{->}_{\varphi_V} "3";"4"};
\endxy
\]  
commutes, where the vertical maps are restrictions.  Hence we get an
isomorphism of presheaves
\[
\varphi: \Lambda^k \scS_E \to \scS_{\Lambda^k E}.
\]  
Since the right hand side is a sheaf, the exterior power $\Lambda^k
\scS_E $ is a sheaf as well.

Now consider $E = T^*M$.  Then
\[
  \bOmega^k_{dR, M} = \scS_{\Lambda^k T^*M} \simeq \Lambda^k
  \scS_{T^*M} = \Lambda^k \bOmega_{dR,M}^1.
\]
Since the presheaf $\Omega^1 _{\cin_M}$ is (isomorphic to)
$\bOmega_{dR,M}^1$ \cite[Proposition~4.7.5]{FS},  the $k$th exterior power $ \Lambda^k \Omega^1
_{\cin_M}$ of $\Omega^1 _{\cin_M}$ is (isomorphic to)  $ \Lambda^k \bOmega_{dR,M}^1 $. The
result follows.
\end{proof}

\begin{lemma}  \label{lem:9.5} A map  $(f,f_\#): (M, \scA)\to (N, \scB)$  of
 local $\cin$-ringed spaces induces a map 
\[
\tilde{f}: \Lambda^\bullet \Omega^1_\scB \to  f_*(\Lambda^\bullet \Omega^1_\scA)
\]
of presheaves of CDGAs.
\end{lemma}

\begin{proof}
The proof is another application of functoriality of $\Lambda^\bullet
\Omega^1_\scC$ in the $\cin$-ring $\scC$, which is Proposition~\ref{prop:6.5}.
For any sheaf $\scC$ of $\cin$-rings  on a space $N$, any $k\in \Z$ and any open set $W\subset
N$
\[
(\Lambda^k \Omega^1_\scC)(W) = \Lambda^k (\Omega^1_{\scC(W)}),
\]
where on the left we evaluate the presheaf on an open set, and on the
right we take the $k$th exterior power of an $\scC(W)$-module
$\Omega^1_{\scC(W)}$.  It follows that for an open set $U\subset N$
\[
  \left(\Lambda^k \Omega^1_{f_*\scA} \right)(U)=
  \Lambda^k \left(\Omega^1_{(f_*\scA)(U)} \right) =
  (\Lambda^k \Omega^1_\scA)(f\inv (U)) =\left( f_*( \Lambda^k \Omega^1_\scA)\right)
  \, (U).
\]  
Consequently,
\begin{equation} \label{eq:9.1}
(\Lambda^\bullet \Omega^1_{\scA (f\inv (U))}, \wedge, d) =\left[ f_*
  (\Lambda^\bullet  \Omega^1_\scA, \wedge, d)\right] (U),
\end{equation}
an equality of CDGAs of modules.
For any pair of open subsets $V\subset U\subset N$ of $N$ the diagram
\[
\xy
(-15,10)*+{\scB(U)}="1";
(15,10)*+{\scA(f\inv(U))}="2";
(-15,-5)*+{\scB(V)}="3";
(15,-5)*+{\scA(f\inv(V))}="4";
{\ar@{->}^{f_\#} "1";"2"};
{\ar@{->}^{} "1";"3"};
{\ar@{->}  "2";"4"};
{\ar@{->}_{f_\#} "3";"4"};
\endxy
\]  
commutes (the vertical maps are the restrictions).  By
Proposition~\ref{prop:6.5}
the diagram 
\[
\xy
(-15,10)*+{\Lambda^\bullet \Omega^1_{\scB(U)}}="1";
(25,10)*+{\Lambda^\bullet \Omega^1_{\scA(f\inv(U))}}="2";
(-15,-5)*+{\Lambda^\bullet \Omega^1_{\scB(V)}}="3";
(25,-5)*+{\Lambda^\bullet \Omega^1_{\scA(f\inv(V))}}="4";
{\ar@{->}^{\Lambda^\bullet(f_\#) } "1";"2"};
{\ar@{->}^{} "1";"3"};
{\ar@{->}  "2";"4"};
{\ar@{->}_{\Lambda^\bullet(f_\#)} "3";"4"};
\endxy
\]  
commutes.  Equation \eqref{eq:9.1} now implies that the maps
$\Lambda^\bullet(f_\#) $ of modules assemble into a desired map $ \tilde{f}$
of presheaves of  CDGAs over the sheaf $\scB$.
\end{proof}

In general there is no reason to expect that the presheaf
$\Lambda^\bullet \Omega^1_\scA$ of CDGAs associated to 
 $\cin$-ringed  space $(M,\scA)$  is a sheaf.  So we need
 to sheafify.  Recall that there is a sheafification functor
 $\sh:\Psh_M\to Sh_M$ from the category of presheaves of
 $\scA$-modules to the category of sheaves.
 The functor is left-adjoint to the inclusion functor
 $i: \Sh_M\hookrightarrow \Psh_M$  (more often than not we will
 suppress the inclusion functor).

\begin{lemma} \label{lem:9.7}
Let $(\{\cM^k\}_{k=0}^\infty, d, \wedge)$ be a presheaf of CDGAs on
a $\cin$-ringed space $(M, \scA)$.  Then
$(\{\sh(\cM^k)\}_{k=0}^\infty, \sh(d), \sh(\wedge))$
is a sheaf of CDGAs on
the $\cin$-ringed space $(M, \scA)$.  Here as before $\sh: \Psh_M\to
\Sh_M$ is the sheafification functor.
\end{lemma} 

\begin{proof}
Note that since the sheafification of
  the zero (pre)sheaf is zero, the sheafification of any zero morphism
  is zero as well: $\sh(0) =0$. Then $0 = \sh(0) = \sh(d^2) =
  (\sh(d))^2$.  Hence $(\{\sh(\cM^k)\}_{k=0}^\infty, 
\sh(d))$ is a cochain complex of sheaves.      Since sheafification
preserves finite products we get 
collection of morphisms
\[
  \sh(\wedge): \sh(\cM^k)  \times \sh(\cM^\ell)\to \sh(\cM^{k+\ell})
\]
We can express the fact that $\wedge$ is graded commutative and the
fact that $d$ is a derivation of degree 1 with respect to $\wedge $
in terms of commutative diagrams.   Since $\sh$ is a
functor, it follows that $(\{\sh(\cM^k)\}_{k=0}^\infty, \sh(d), \sh(\wedge))$  is a sheaf
of CDGAs.  
\end{proof}

\begin{definition} \label{def:9.8}
We define the {\sf $\cin$-algebraic de Rham complex
$(\bOmega^\bullet_\scA, \sh(\wedge), \sh(d))$ of sheaves} on a
$\cin$-ringed space $(M,\scA)$ to be the sheafification of the
presheaf of CDGAs $(\{\Lambda^k\Omega^1_\scA\} _{k\geq 0}, \wedge,
d)$, which was constructed in Lemma~\ref{constr:9.3}.
  \end{definition}

Recall that 
given a manifold with corners $M$ we have the de Rham complex $\bOmega^\bullet_{dR,
  M}$ of sheaves of ordinary differential forms and the $\cin$-algebraic de Rham
complex $\bOmega_{\cin_M}^\bullet $ of the manifold $M$ viewed as
local $\cin$-ringed space $(M, \cin_M)$.  By
Lemma~\ref{lem:CDGAcorners} the two sheaves of CDGAs agree.

Differential forms pull back.  Our $\cin$-algebraic differential forms on
$\cin$-ringed spaces had better pull back as well: given a map
$(f,f_\#): (M, \scA)\to (N, \scB)$ of $\cin$-ringed spaces we expect
to have a
map of sheaves of CDGAs $\tilde{f}:\bOmega^\bullet _\scA\to
f_*(\bOmega^\bullet _\scB)$.  This leads to a complication.
By Lemma~\ref{lem:9.5} a map $(f,f_\#): (M, \scA)\to (N, \scB)$  of local
  $\cin$-ringed spaces induces a map $
\tilde{f}: \Lambda^\bullet \Omega^1_\scB \to  f_*(\Lambda^\bullet \Omega^1_\scA)$
of presheaves of CDGAs over $\scB$.  Recall that the adjunction $\sh_N
\dashv i_N$ comes with the natural transformation
$\eta:\id_{\Psh_N}\Rightarrow i\circ \sh_N$, the unit of adjunction.
The components $\eta_\scP :\scP \to \sh_N(\scP)$ of the natural
transformation $\eta$ have the following
universal property: for any sheaf $\scS$ and any map of presheaves
$\varphi:\scP\to \scS$ there is a unique map of presheaves
$\bar{\varphi}: \sh_N(\scP)\to \scS$ so that
\[
\bar{\varphi}\circ \eta_\scP = \varphi.
\]  
The same holds for presheaves and sheaves of $\scA$-modules over $M$.
Since the pushforward $f_*:\Psh_M\to \Psh_N$ is a functor, we have a
map of presheaves
\[
f_*(\Lambda^k\Omega^1_\scA) \xrightarrow{ f_*
  (\eta_{\Lambda^k\Omega^1_\scA}) }
f_* (\bOmega^k_\scA).
\]  
Consequently we get the following diagram of maps of presheaves
\[
\bOmega^k_\scB \xleftarrow{ \eta_{ \Lambda^k\Omega^1_\scB} } \Lambda^k\Omega^1_\scB
\xrightarrow{\quad \tilde{f} \quad }
f_*(\Lambda^k\Omega^1_\scA) \xrightarrow{ f_* (\eta_{\Lambda^k\Omega^1_\scA}) } f_* (\bOmega^k_\scA),
\]
which we can't compose since one of the maps goes in the wrong
direction.  However the universal property of the map $\eta_{
  \Lambda^k\Omega^1_\scB}$
gives us $\bar{f}^k: \bOmega^k_\scB\to \bOmega^k_\scA$ making the
diagram
\[
\xy
(-35,15)*+{\bOmega^k_\scB}="1";
(40,15)*+{f_*\bOmega^k_\scA}="2";
(-15,0)*+{\Lambda^k \Omega^1_\scB}="3";
(15,-0)*+{f_*(\Lambda^k \Omega^1_\scA)}="4";
{\ar@{->} ^{\bar{f}^k} "1";"2"};
{\ar@{->}^{\eta_{\Lambda^k\Omega^1_\scB}} "3";"1"};
{\ar@{->} _{\eta_{\Lambda^k\Omega^1_\scA}}"4";"2"};
{\ar@{->}_{\tilde{f}} "3";"4"};
\endxy
\]  
commute.  And now we have a problem: it is not clear that the {\em outer}
square in the diagram
\[
\xy
(-35,15)*+{\bOmega^k_\scB}="1";
(40,15)*+{f_*\bOmega^k_\scA}="2";
(-15,0)*+{\Lambda^k \Omega^1_\scB}="3";
(15,-0)*+{f_*(\Lambda^k \Omega^1_\scA)}="4";
(-35,-30)*+{\bOmega^{k+1}_\scB}="11";
(40,-30)*+{f_*\bOmega^{k+1}_\scA}="22";
(-15,-15)*+{\Lambda^{k+1} \Omega^1_\scB}="33";
(15,-15)*+{f_*(\Lambda^{k+1} \Omega^1_\scA)}="44";
{\ar@{->} ^{\bar{f}^k} "1";"2"};
{\ar@{->} _{d} "3";"33"};
{\ar@{->} ^{d} "4";"44"};
{\ar@{->}^{\eta_{\Lambda^k\Omega^1_\scB}} "3";"1"};
{\ar@{->} _{\eta_{\Lambda^k\Omega^1_\scA}}"4";"2"};
{\ar@{->}^{\tilde{f}} "3";"4"};
{\ar@{->}_{\bar{f}^k} "11";"22"};
{\ar@{->}^{\eta_{\Lambda^{k+1}\Omega^1_\scB}} "33";"11"};
{\ar@{->} _{\eta_{\Lambda^{k+1}\Omega^1_\scA}}"44";"22"};
{\ar@{->}_{\tilde{f}} "33";"44"};
{\ar@{->}_{\sh(d)} "1";"11"};
{\ar@{->} ^{\sh(d)} "2";"22"};
\endxy
\]  
commutes.  The same issue arises with compatibility of multiplications
on $\bOmega^\bullet _\scB$ and $f_*\bOmega^\bullet _\scA$: it is not
clear that $\bar{f}: \bOmega^\bullet _\scB \to f_*\bOmega^\bullet
_\scA$ is a map of sheaves of CGAs.   For this reason we carry out  a
category-theoretic exercise first.

\begin{lemma} \label{lem:9.9}
Let $\xy
  (-8,0)*+{\scC}="1";
  (8,0)*+{\scD}="2";
  (0,0)*+{\perp};
  {\ar@/^0.7pc/@{->}^{L} "1"; "2"};
  {\ar@/^0.7pc/@{->}^{R} "2"; "1"};
  \endxy$
  be a pair of adjoint functors, $\eta: \id_\scC\Rightarrow
  RL$ the unit of adjunction, $h:d\to d'$ a morphism in $\scD$,
  $g:c\to c'$, $k:c\to Rd$, $k': c'\to Rd'$ morphisms in $\scC$ so
  that the diagram
  \begin{equation}\label{eq:9.2}
\xy
(-15,10)*+{c}="1";
(15,10)*+{Rd}="2";
(-15,-5)*+{c'}="3";
(15,-5)*+{Rd'}="4";
{\ar@{->}^{k} "1";"2"};
{\ar@{->}_{g} "1";"3"};
{\ar@{->} ^{Rh}  "2";"4"};
{\ar@{->}_{k'} "3";"4"};
\endxy
\end{equation}
commutes in $\scC$.    Let $\tilde{k}:Lc\to d$, $\tilde{k'}:Lc'\to d'$ be the unique
morphisms with $R\tilde{k} \circ \eta_c = k$, $R\tilde{k'}\circ
\eta_{c'} = k'$.  Then the diagram
\[
\xy
(-15,10)*+{Lc}="1";
(15,10)*+{d}="2";
(-15,-5)*+{Lc'}="3";
(15,-5)*+{d'}="4";
{\ar@{->}^{\tilde{k}} "1";"2"};
{\ar@{->}_{Lg} "1";"3"};
{\ar@{->} _{h}  "2";"4"};
{\ar@{->}_{\tilde{k'}} "3";"4"};
\endxy
\]
commutes in $\scD$. 
\end{lemma}  

\begin{proof}
Since $L$ is left adjoint to $R$, for all $a\in \scC$, $b\in \scC$
there are bijections $\theta_{a,b}:\Hom_\scC(a, Rb)\to \Hom_\scD(La,
b)$ which are natural in $a$ and $b$.  Moreover
\[
(\theta_{a,b})\inv (\tilde{\ell}) = R(\tilde{\ell})
\circ \eta _a
\]
for all $\tilde{l} \in \Hom_\scD(La,
b)$.  Hence $\tilde{k} = \theta_{c,d}(k)$ and $\tilde{k'} = \theta_{c',d'}(k')$.
Since $\theta_{a,b}$ is natural in $a$ and $b$ the
diagram
\begin{equation}\label{eq:88*}
\xy
(-15,10)*+{\Hom(c, Rd)}="1";
(25,10)*+{\Hom(Lc,d)}="2";
(-15,-5)*+{\Hom(c, Rd')}="3";
(25,-5)*+{\Hom(Lc,d')}="4";
(-15,-20)*+{\Hom(c', Rd')}="5";
(25,-20)*+{\Hom(Lc',d')}="6";
{\ar@{->}^{\theta_{c,d}} "1";"2"};
{\ar@{->}_{(Rh)_*} "1";"3"};
{\ar@{->} ^{h_*}  "2";"4"};
{\ar@{->}_{\theta_{c,d'}} "3";"4"};
{\ar@{<-}_{g^*} "3";"5"};
{\ar@{<-}_{(Lg)^*} "4";"6"};
{\ar@{->}_{\theta_{c',d'}} "5";"6"};
\endxy
\end{equation}
commutes.  We now compute:
\[
h\circ \tilde{k} \overset{\eqref{eq:88*}}{=} h_*\theta_{c,d}(k)  = \theta_{c,d'} ((Rh)_* (k))
\overset{\textrm{  \eqref{eq:9.2} }}{=} \theta_{c,d'}
(k'\circ g)   \overset{\eqref{eq:88*}}{=} (Lg)^* \theta_{c',d'}(k') = \tilde{k'} \circ Lg.
\]
Therefore $h\circ \tilde{k} = \tilde{k'} \circ Lg$ and we are done.
\end{proof}

\begin{corollary} \label{cor:9.10}
Let $A_1\xrightarrow{h} A_2$ be a map of presheaves over a space $M$,
$B_1\xrightarrow{g} B_2$ a map of presheaves over a space $N$, $f:M\to
N$ a continuous map and $f_i: B_i \to f_* A_i$, $i=1,2$ maps of
presheaves making the diagram 
\[
\xy
(-15,10)*+{B_1}="1";
(15,10)*+{f_*A_1}="2";
(-15,-5)*+{B_2}="3";
(15,-5)*+{f_*A_2}="4";
{\ar@{->}^{f_1} "1";"2"};
{\ar@{->}_{g} "1";"3"};
{\ar@{->} ^{f_*h} "2";"4"};
{\ar@{->}_{f_2} "3";"4"};
\endxy
\]
commute.   Let $\eta_{A_i}:A_i\to \sh(A_i)$, $\eta_{B_i}:B_i\to
\sh(B_i)$ denote the sheafifications and let $\tilde{f}_i: \sh(B_i)\to
f_*(\sh(A_i))$ be the canonical maps with
\[
\tilde{f}_i \circ \eta_{B_i} = f_* (\eta_{A_i}) \circ f_i, \quad i=1,2.
\]
Then the diagram
\[
\xy
(-15,10)*+{\sh(B_1)}="1";
(15,10)*+{f_*\sh(A_1)}="2";
(-15,-5)*+{\sh(B_2)}="3";
(15,-5)*+{f_*\sh(A_2)}="4";
{\ar@{->}^{\tilde{f}_1} "1";"2"};
{\ar@{->}_{\sh(g)} "1";"3"};
{\ar@{->} ^{f_*\sh(h)} "2";"4"};
{\ar@{->}_{\tilde{f}_2} "3";"4"};
\endxy
\]
commutes.
\end{corollary}
\begin{proof}
We know that the sheafification functor $\sh$ is left adjoint to the
inclusion $i:\Sh\hookrightarrow \Psh$.   Apply Lemma~\ref{lem:9.9}  to
the commutative diagram
\[
\xy
(-25,10)*+{B_1}="1";
(15,10)*+{i(f_*\sh(A_1))}="2";
(-25,-5)*+{B_2}="3";
(15,-5)*+{i(f_*\sh(A_2))}="4";
{\ar@{->}^{f_* (\eta_{A_1}) \circ f_1} "1";"2"};
{\ar@{->}_{g} "1";"3"};
{\ar@{->} ^{i(f_*\sh(h))} "2";"4"};
{\ar@{->}_{f_* (\eta_{A_2}) \circ f_2} "3";"4"};
\endxy .
\]
\end{proof}

We are now in position to prove that $\cin$-algebraic differential
forms pull back along maps of $\cin$-ringed spaces, and that the pullback
commutes with exterior multiplication and exterior derivatives.
\begin{theorem} \label{thm:8.13}
 A map  $(f,f_\#): (M, \scA)\to (N, \scB)$  of
  $\cin$-ringed spaces induces a map 
\[
\bar{f}: \bOmega^\bullet _\scB \to  f_*(\bOmega^\bullet_\scA)
\]
of sheaves of CDGAs.
\end{theorem}
\begin{proof}
By Lemma~\ref{lem:9.5}  the map $(f, f_\#)$ induces a map $\tilde{f}:
\Lambda^\bullet \Omega^1_\scB \to  f_*(\Lambda^\bullet \Omega^1_\scA)$
of presheaves of CDGAs.   In particular for each $k\geq 0$ we have
maps $\tilde{f}: \Lambda^k \Omega^1_\scB \to  f_*(\Lambda^k
\Omega^1_\scA)$ so that the following two diagrams
\[
\xy
(-25,10)*+{\Lambda^k \Omega^1_{\scB}\times \Lambda^\ell \Omega^1_{\scB}}="1";
(25,10)*+{f_*\Lambda^k \Omega^1_{\scA}\times f_*\Lambda^\ell \Omega^1_{\scA}}="2";
(-25,-5)*+{\Lambda^{k+\ell} \Omega^1_{\scB}}="3";
(25,-5)*+{f_*\Lambda^{k+\ell} \Omega^1_{\scA}}="4";
{\ar@{->}^{\tilde{f} \times \tilde{f} } "1";"2"};
{\ar@{->}_{\wedge} "1";"3"};
{\ar@{->}^{\wedge} "2";"4"};
{\ar@{->}_{\tilde{f} } "3";"4"};
\endxy
\]
\[
\xy
(-15,10)*+{\Lambda^k \Omega^1_{\scB}}="1";
(25,10)*+{f_*\Lambda^k \Omega^1_{\scA}}="2";
(-15,-5)*+{\Lambda^{k+1} \Omega^1_{\scB}}="3";
(25,-5)*+{f_*\Lambda^{k+1} \Omega^1_{\scA}}="4";
{\ar@{->}^{\tilde{f} } "1";"2"};
{\ar@{->}_{d} "1";"3"};
{\ar@{->}^{d} "2";"4"};
{\ar@{->}_{\tilde{f} } "3";"4"};
\endxy
\]  
commute.   Recall the notation: $\bOmega^k_\scA:=
\sh(\Lambda^k\Omega^1_\scA)$ etc.  Apply Corollary~\ref{cor:9.10} to
the two diagrams above and set
\[
  \bar{f} := \tilde{\tilde{f}}
\]
(the first  $\tilde{\,}$ on $f$ comes from Lemma~\ref{lem:9.5} and the
second from Corollary~\ref{cor:9.10}). We get two commuting diagrams:
\[
\xy
(-15,10)*+{\bOmega^k _{\scB}\times \bOmega ^\ell _{\scB}}="1";
(25,10)*+{f_* \bOmega^k_{\scA}\times f_*\bOmega^\ell_{\scA}}="2";
(-15,-5)*+{\bOmega^{k+\ell} _{\scB}}="3";
(25,-5)*+{f_*\bOmega^{k+\ell} _{\scA}}="4";
{\ar@{->}^{\bar{f} \times \bar{f}} "1";"2"};
{\ar@{->}_{\sh(\wedge)} "1";"3"};
{\ar@{->}^{\sh(\wedge)} "2";"4"};
{\ar@{->}_{\bar{f} } "3";"4"};
\endxy
\quad \textrm{and} \qquad 
\xy
(-15,10)*+{\bOmega^k _{\scB}}="1";
(25,10)*+{f_*\bOmega^k _{\scA}}="2";
(-15,-5)*+{\bOmega^{k+1} _{\scB}}="3";
(25,-5)*+{f_*\bOmega^{k+1} _{\scA}}="4";
{\ar@{->}^{\bar{f} } "1";"2"};
{\ar@{->}_{\sh(d)} "1";"3"};
{\ar@{->}^{\sh(d)} "2";"4"};
{\ar@{->}_{\bar{f} } "3";"4"};
\endxy,
\]
for all $k, \ell \geq 0$.
This proves the theorem.
\end{proof}

\begin{corollary} \label{cor:5.11}
Let $M$ be a manifold with corners and $(N, \scB)$ a local
$\cin$-ringed space.  Then a map $(f,f_\#): (M, \cin_N)\to (N, \scB)$
of  $\cin$-ringed spaces induces a map of commutative differential
graded algebras
\[
f^*: \bOmega^\bullet _\scB(N) \to \bOmega^\bullet_{dR,M}(M)
\]
over the $\cin$-ring $\scB(N)$.
\end{corollary}

\begin{proof}
  By Lemma~\ref{lem:CDGAcorners} the sheaf 
  $\bOmega^\bullet_{\cin_M}  $ of $\cin$-algebraic de Rham forms is
  the sheaf $
\bOmega^\bullet_{dR, M}$ of ordinary differential forms. By
Theorem~\ref{thm:8.13} the map $(f,f_\#)$ gives rise to a map
$\bar{f}:\bOmega^\bullet_\scB\to f_* \bOmega^\bullet_{dR, M}$ of sheaves of CDGAs, which is a
collection of maps of CDGAs
\[
\{(\bar{f})_U: \bOmega^\bullet_\scB(U) \to \bOmega^\bullet_{dR,
  M}(f\inv (U))\}_{U\in \Open(N)}.
\]
For $U= N$ we get
\[
f^*: = (\bar{f})_N: \bOmega^\bullet_\scA(N) \to \bOmega^\bullet_{dR,
  M}(M) \equiv \bOmega^\bullet_{dR}(M).
\]  
\end{proof}

\begin{remark} \label{rmrk:5.12} Let $M$ be a manifold with corners
  and $(f,f_\#): (M, \cin_M)\to (N, \scB)$ a map $\cin$-ringed spaces
  as in Corollary~\ref{cor:5.11}.  Suppose now that $N$ is a manifold
  with corners and the structure sheaf $\scB$ is the sheaf $\cin_N$ of
  smooth functions.  Then we can view both $N$ and $M$ as differential
  spaces.  Since the category of differential spaces embeds into the
  category local $\cin$-ringed spaces  (Theorem~\ref{thm:8.2}) the map
  $f:M\to N$ is a smooth map (in
  the sense that for any $h\in \cin(N)$, $h\circ f \in \cin(M)$; there
  are other, more restrictive definitions of smooth maps between
  manifolds with corners).  Unravelling various definitions one
  discovers that the map
  $f^*: \bOmega^\bullet _{\cin_N} (N) \to \bOmega^\bullet_{dR}(M)$ is
  the usual pullback  map
  $f^* :\bOmega^\bullet _{dR}(N) \to \bOmega^\bullet_{dR}(M)$ of
  ordinary differential forms.  This uses the facts that
  $\bOmega^\bullet_{dR}(M) = \Lambda^\bullet \bOmega^1_{dR}(M)$, that
  $\Omega^1_{\cin(M)} = \bOmega^1_{dR}(M)$ and that $\Lambda^1(f^*)
  :\Omega^1_{\cin(N)} \to \Omega^1_{\cin(M)}$ induced by the map
  $f^*:\cin(N)\to \cin (M)$ of $\cin$-rings agrees with the pullback
  map $f^* : \bOmega^1_{dR}(N) \to \bOmega^1_{dR}(M)$ on 1-forms and so on.
\end{remark}

\section{Integration  and Stokes' theorem}

Recall from the previous section that for any manifold with corners $Q$ the presheaf
$(\Lambda^\bullet \Omega^1_{\cin_Q}), d, \wedge)$ of $\cin$-algebraic
de Rham forms agrees with the sheaf
$(\bOmega^\bullet_{dR,Q}, d, \wedge)$ of ordinary differential forms
(Lemma~\ref{lem:CDGAcorners}).  Consequently sheafification of
$\Lambda^\bullet \Omega^1_{\cin_Q}$ does nothing and the possibly
mysterious sheaf
$\bOmega^\bullet_{\cin_Q}= \sh(\Lambda^\bullet \Omega^1_{\cin_Q})$
turns out to be 
the sheaf $(\bOmega^\bullet_{dR,Q}, d, \wedge)$ of ordinary de Rham
differential forms.  Next we recall a definition of the standard
(geometric) $k$-simplex.

\begin{definition} \label{def:face}
  The {\sf standard $k$-simplex $\Delta^k$} is the
space
\[
\Delta^k =\{x\in \R^{k+1} \mid x_i\geq 0, x_1+\cdots+ x_k = 1\}.
\]
The $k$-simplex is a manifold with corners, hence a local
$\cin$-ringed space $(\Delta^k, \cin_{\Delta^k})$

The {\sf $i$th face}  of the $k$th simplex is the map $d_i:\Delta^{k-1} \to
\Delta^k$ given by 
\[
d_i(x_1, \ldots, x_{k-1} ) = (x_1, \ldots,
x_{i-1}, 0, x_i,\ldots, x_{k-1}).
\]
We view  $d_i$ as  a map of differential spaces (and of local
$\cin$-ringed spaces, see Remark~\ref{rmrk:6.2}).
\end{definition}

\begin{remark} (Cf.\ Remark~\ref{rmrk:5.12}.) \label{rmrk:6.2}
  Recall that there is a fully faithful functor from the category of
  differential spaces to the category $\LCRS$ of local $\cin$-ringed
  spaces (Theorem~\ref{thm:8.2}).  For this reaon we will not
  notationally distinguished between the map $d_i: \Delta^{k-1} \to
  \Delta^k$ of differential spaces 
  and the corresponding map
  $d_i: (\Delta^{k-1}, \cin_{\Delta^{k-1}}) \to (\Delta^k,
  \cin_{\Delta^k})$ in $\LCRS$. 
\end{remark}  

\begin{definition}
A {\sf $k$-simplex} of a local  $\cin$-ringed space $(M,\scA)$ is a map
$\sigma: (\Delta^k, \cin_{\Delta^k})\to (M,\scA)$ of local
$\cin$-ringed spaces.

We define the {\sf vector space $C_k(M, \scA;\R )$ of real $k$-chains}
of a local $\cin$-ringed space $(M,\scA)$ to be the real vector space
generated by the set of $k$-simplexes of $(M,\scA)$.  We will refer to
elements of $C_k(M, \scA;\R)$ a $k$-chains and write them as finite
 sums $\sum n_{\sigma,} \sigma$.
\end{definition}

\begin{remark}
A simplex $\sigma: (\Delta^k, \cin_{\Delta^k})\to (M,\scA)$ is really
a pair of maps: a continuous map $f:\Delta^k \to M$ and a map of
sheaves $f_\#: \scA \to f_* \cin_{\Delta_k}$.  However, writing $(f,
f_\#)$ for $\sigma$ in this situation is notationally too heavy, so we
will avoid it.
\end{remark}  

\begin{definition}
The {\sf boundary} $\partial \sigma$ of a $k$-simplex $\sigma:
(\Delta^k, \cin_{\Delta^k})\to (M,\scA)$ is the chain $\partial \sigma
\in C_{k-1}(M,\scA;\R)$ defined by 
\[
\partial \sigma = \sum (-1)^i \sigma \circ d_i
\]  
where $d_i:\Delta_{k-1} \to \Delta_k$, $i=1, \ldots, k$ are the
standard face maps of Definition~\ref{def:face}.
\end{definition}
\begin{definition} The {\sf boundary } $\partial \Delta^k$ of the standard
$k$-simplex is the chain $\partial (\id)\in C_{k-1}(\Delta^k,
\cin_{\Delta_k}; \R)$ where $\id: (\Delta^k,\cin_{\Delta_k} )\to (\Delta^k,
\cin_{\Delta_k})$ is the identity $k$-simplex.
\end{definition}

\begin{definition}[Integral of a global $\cin$-algebraic differential
  form]
Let  $(M,\scA)$ be a local $\cin$-ringed space, $\sigma: (\Delta^k,
\cin_{\Delta^k})\to (M,\scA)$ a $k$-simplex and $\omega \in \bOmega^k_\scA
(M) = \left(\sh (\Lambda^k \Omega^1_\scA) \right)(M)$ a global $\cin$-algebraic differential
form.   Then, since by Lemma~\ref{lem:CDGAcorners},  $\sh(\Lambda^k \Omega^1_{\cin_{\Delta^k}})
=\bOmega^k_{dR} (\Delta^k)$, the pullback $\sigma^*\omega $ is an
ordinary $k$-form on the manifold with corners $\Delta^k$.

We define
{\sf the integral of $\omega$ over $\sigma$} to be
\[
\int_\sigma\omega := \int_{\Delta^k} \sigma^*\omega.
\]
We define {\sf the integral of $\omega$ over a $k$-chain }
$\sum n_\sigma \sigma\in C_k (M,\scA;\R)$ to be
\[
\int_{\sum n_\sigma \sigma} \omega := \sum n_\sigma  \int_{\Delta^k} \sigma^*\omega.
\]
\end{definition}   

Recall Stokes's theorem for simplices (see, for example, \cite[Theorem~8.2.9]{Conlon}) :

\begin{theorem} \label{thm:Stokes_simp} For any $k>0$ and for any differential form
  $\gamma\in \bOmega_{dR}(\Delta^k)$ 
\[
   \int_{\Delta_k} d\gamma = \int_{\partial \Delta_k} \gamma.
\]
\end{theorem}
As an immediate consequence we get
\begin{theorem}
Let $(M,\scA)$ be a local $\cin$-ringed space, $\gamma\in
\bOmega^{k-1}_\scA(M)$ a global $\cin$-algebraic de Rham  form
and $\sigma: (\Delta^k, \cin_{\Delta^k})\to (M,\scA)$ a $k$-simplex.
Then
\[
\int_\sigma d\gamma = \int_{\partial \sigma} \gamma.
\]  
\end{theorem}
\begin{proof}
The proof is a computation:
\[
\int_\sigma d\gamma = \int_{\Delta^k} \sigma^*(d\gamma) =
\int_{\Delta^k} d (\sigma^*\gamma) \stackrel{\textrm{Theorem~\ref{thm:Stokes_simp}}}{=} \int_{\partial \Delta^k} \sigma^*
\gamma = \sum (-1)^i \int_{\Delta^{k-1}} d_i^* (\sigma^*\gamma) = \int_{\partial
  \sigma} \gamma.
\]  
\end{proof}  

\begin{remark}
Instead of integrating over singular chains we may just as easily
integrate over compact manifolds with boundary and then there is a
version of the Stokes' theorem as well. 
\end{remark}  

\appendix

\section{Differential spaces and bump functions} \label{app:bump}
\noindent The goal  of the appendix is to prove that condition
(\ref{def:sikorksi:it1}) in the definition of a differential space is
equivalent to existence of bump functions.  The material  is extracted from \cite{KL}.

\begin{definition} \label{def:a.30} Let $(M,\scT)$ be a topological space,
  $C\subset M$ a closed set and $x\in M\smallsetminus C$ a point.  A
  {\sf bump function} (relative to $C$ and $x$) is a continuous
  function $\rho:M\to [0,1]$ so that $(\supp \rho)\, \cap C = \varnothing $
  and $\rho$ is identically 1 on a neighborhood of $x$.
\end{definition}

\begin{definition} \label{def:a.31} Let $(M,\scT)$ be a topological
  space and $\scF \subseteq C^0(M,\R)$ a collection of continuous
  real-valued functions on $M$. The topology $\scT$ on $M$ is {\sf
    $\scF$-regular} iff for any closed subset $C$ of $M$ and any point
  $x\in M\smallsetminus C$ there is a bump function $\rho\in \scF$
  with $\supp \rho \subset M\smallsetminus C$ and $\rho$ identically 1
  on a neighborhood of $x$.
\end{definition}  

\begin{lemma}\label{lem:bump}
Let $(M,\scT)$ be a topological space and $\scF \subset
  C^0(M,\R)$ a $\cin$-subring. Then $\scT$ is the smallest topology making all
  the functions in $\scF$ continuous if and only if the topology $\scT$ is $\scF$-regular.
\end{lemma}  
 
\begin{proof}
Let $\scT_\scF$ denote the smallest topology making all the functions
in $\scF$ continuous.  Note that the set
\[
 \{ f^{-1}(I) \mid  f \in \scF,  \text{ $I$ is an open
    interval } \}
\]  
is a sub-basis for $\scT_\scF$.  Since all the functions in $\scF$ are
continuous with respect to $\scT$, $\scT_\scF\subseteq \scT$.
Therefore it is enough to argue that $\scT\subseteq \scT_\scF $ if and
only if $\scT$ is $\scF$-regular.\\[4pt]
($\Rightarrow$)\quad Suppose $\scT \subseteq \scT_\scF$.  Let
$C\subset M$ be $\scT$-closed and $x$ a point in $M$ which is not in
$C$.  Then $M \setminus C$ is $\scT$-open.  Since $\scT \subseteq
\scT_\scF$ by assumption, $M \setminus C$ is in $\scT_\scF$.  
Then there exist functions $h_1,\ldots,h_k \in \scF$
and open intervals $I_1,\ldots,I_k$ such that
$x \in \cap_{i=1}^k h_i^{-1}(I_i) \subset M \setminus C$.  There is a
$\cin$ function $\rho: \R^k \to [0,1]$
with $\supp \rho \subset I_1 \times \ldots \times I_k$ and the
property that
$\rho = 1$ on a neighborhood of $(h_1(x),\ldots,h_k(x))$ in $\R^k$.
Then $\tau: = \rho \circ (h_1,\ldots,h_k)$ is in $ \scF$, since $\scF$ is a
$\cin$-subring of $C^0(M)$.  The function $\tau$ is a desired bump
function.\\[4pt]
($\Leftarrow$)\quad Suppose the topology $\scT$ is
$\scF$-regular. Let $U\in \scT$ be an open set.   Then $C= M\setminus
U$ is closed.  Since $\scT$ is $\scF$-regular, for any point $x\in U$ there
is a bump function $\rho_x \in \scF$ with $\supp \rho_x \subset U$ and
$\rho_x $ is identically $1$ in a neighborhood of $x$.   Then $\rho_x\inv ((0,\infty)) \subset U$
and $\rho_x\inv ((0,\infty)) \in \scT_\scF$.   It follows that
\[
U = \bigcup_{x\in U} \rho_x\inv ((0,\infty)) \in \scT_\scF.
\]  
Since $U$ is an arbitrary element of $\scT$, $\scT \subseteq \scT_\scF$.
\end{proof}

\section{Differential spaces as   local $\cin$-ringed
  spaces} \label{app:embed}

The goal of the appendix is to prove the following theorem.

\begin{theorem}\label{thm:8.2}
The category $\DiffSp$ of differential spaces is (isomorphic to) a
full subcategory of the category $\LCRS$ of local $\cin$-ringed spaces.
\end{theorem}

\begin{proof}
We construct a functor $I: \DiffSp \to \LCRS$ and check that it is
fully faithful.  We define $I$ on objects by sending a differential
space $(M, \scF)$ to the corresponding local $\cin$-ringed space
$(M, \scF_M)$:
\[
I( (M, \scF ))= (M, \scF_M),
\]
cf.\ Lemma~\ref{lem:2.58}.

We need to define $I$ on morphisms.  Recall that a smooth map
$f: (M, \scF)\to (N, \scG)$ between two differential spaces, i.e., a
morphism in $\DiffSp$, is a continuous map $f:M\to N$ with the
property that for all $h\in \scG$ its pullback $f^*h$ is in $\scF$.
To define $I(f)$ we need to construct a map of sheaves
$f_\#: \scG_N\to f_* \scF_M$ and set $I(f):= (f, f_\#)$.

Recall that $\scG_N$ is the sheafification of the presheaf $\cP_N$
defined by
\[
\cP_N(U) = \scG|_U
\]  
for all open sets $U\subseteq N$.   Since pullback by $f$ commutes
with restrictions we get a map of presheaves
\[
f_\#: \cP_N \to f_* (\cP_M) \hookrightarrow f_* \scF_M, \qquad f_\# (h) = h\circ f.
\]  
The universal property of sheafification $\cP_N\hookrightarrow \scG_N$
uniquely extends $f_\#$ to a map
\[
\scG_N\to f_*\scF_M, 
\]
which is given by the same formula: for any open set $U\subset N$ and
for any function $h\in \cin_N(U)$,
\[
h\mapsto h\circ f \in \scF_M(f\inv (U)).
\]
We now abuse notation and define $f_\#: \scG_N\to f_*\scF_M$  by
$f_\#(h):= h\circ f$.   The map $f_\#$ is a map of sheaves.
We now define $I$ on morphisms by
\[
I\big( (M, \scF) \xrightarrow{f} (N, \scG)\big) = (f, f_\#): (M,
\scF_M)\to (N, \scG_N), 
\]  
where as above $f_\# h = h\circ f$.

It is not hard to check that $I$ preserves identity maps and
composition.  Therefore $I$ is a functor.
If $I(f) = (f, f_\#) = (g, g_\#) = I(g)$ for two maps $f, g: (M,
\scF) \to (N, \scG)$ then clearly $f=g$.   Therefore $I$ is
injective on $\Hom$ sets, i.e, the functor $I$ is faithful.

It remains to prove that $I$ is surjective on $\Hom$ sets.    This is
where locality has to come in.  Suppose we have a continuous map
$f:M\to N$ and a map of sheaves $\psi: \scG_N \to f_*\scF_M$.  We want
to show that for any open set $U\subset N$ and any $h\in \scG_N(U)$
\[
\psi (h) = h\circ f,
\]  
for then $(f, \psi) = I(f)$.

Consider a point $p\in U$ and a point $q\in f\inv (p)$.  The map
$\psi$ induces a map on stalks
\[
\psi_p: (\scG_N)_p \to (f_*\scF_M)_p.
\]
Since $q\in f\inv (p)$, there is a canonical map of stalks
$(f_*\scF_M)_p \to (\scF_M)_q$.  Precomposing this canonical map with
$\psi_p$ we get a map
\[
\varphi: (\scG_N)_p\to (\scF_M)_q.
\]  
We can give this map $\varphi$ a concrete description: if $U$ is a
neighborhood of $p$ and $h\in \scG_N(U)$ then
\[
\varphi ([(U, h)] = [(f\inv (U), \psi (h))] \in (\scF_M)_q
\]  
The stalk $(\scF_M)_q$ is a local $\cin$-ring, hence has a unique
$\R$-point $\hat{q}: (\scF_M)_q \to \R$: $\hat{q} (g_q) = g(q)$ for
all germs $g_q\in (\scF_M)_q$;   it is the evaluation at $q$.  The composite
$\hat{q} \circ \varphi:(\scG_N)_p \to \R$ is also an $\R$-point.
Since the $\cin$-ring $(\scG_N)_p$ is local, it has only one
$\R$-point, namely $\hat{p}$, the evaluation at $p$.  Therefore
\[
\hat{q} \circ \varphi = \hat{p}.
\]  
We are almost done.
Recall that any nonzero $\cin$-ring $\scC$ contains a copy of $\R$.  This
copy correspond to operations of arity $0$, i.e., to constants, and any
map of $\cin$-rings preserves these constants.
Now for any germ $h_p \in (\scF_M)_p$
\[
0 = \hat{p} (h_p - h(p))  = \hat{q}(\varphi(h_p -h(p))) = \hat{q} 
  (\varphi(h_p) ) - h(p) 
\]  
since $h(p)$ is a constant.  Since $h(p) = h(f(q))$ and since $\hat{q} 
  (\varphi(h_p) ) = \psi (h) (q)$  we have
\[
\psi (h) (q) = h (f(q))
\]
for all $q\in f\inv (U)$, and we are done. 
\end{proof}

\end{document}